\documentclass{article}
\usepackage[utf8]{inputenc}
\usepackage{fullpage}
\usepackage{amssymb}
\usepackage{amsmath}
\usepackage{amsthm}
\usepackage{enumerate}
\usepackage{CJK}
\usepackage{color}
\usepackage{mathrsfs}
\usepackage{array}
\newtheorem{theorem}{Theorem}[section]
\newtheorem{proposition}[theorem]{Proposition}
\newtheorem{definition}[theorem]{Definition}
\newtheorem{corollary}[theorem]{Corollary}
\newtheorem{lemma}[theorem]{Lemma}
\newtheorem{remark}[theorem]{Remark}
\numberwithin{equation}{section} \numberwithin{theorem}{section}
\numberwithin{equation}{section}

\newtheorem{example}[theorem]{Example}
\numberwithin{equation}{section}

\usepackage{cancel} 
\usepackage{ulem}
\usepackage{soul}

\newcommand{\R}{{\mathbb{R}}}
\newcommand{\B}{{\mathbb{B}}}
\newcommand{\Sph}{{\mathbb{S}}}

\DeclareMathOperator*\argmin{arg\,min}

\DeclareMathOperator*\cl{cl}
\DeclareMathOperator*\gph{gph}
\DeclareMathOperator*\epi{epi}

\DeclareMathOperator*\inte{int}
\DeclareMathOperator*\ri{ri}

\DeclareMathOperator*\dom{dom}

\DeclareMathOperator*\sign{sign}

\DeclareMathOperator*\supp{supp}

\usepackage{amssymb}
\usepackage{mathrsfs}

\title{Variational Analysis of Kurdyka-{\L}ojasiewicz Property by Way of Outer Limiting Subgradients}

\author{Minghua Li\thanks{School of Mathematics and Big Data, Chongqing University of Arts and Sciences, Yongchuan, Chongqing, 402160, China (minghuali20021848@163.com).} 
\and Kaiwen Meng\thanks{School of Mathematics, Southwestern University of Finance and Economics, Chengdu 611130, China (mengkw@swufe.edu.cn)}
\and Xiaoqi Yang\thanks{Department of Applied Mathematics, The Hong Kong Polytechnic University, Hong Kong  (mayangxq@polyu.edu.hk).}}

\begin{document}
\maketitle

\begin{abstract}
In this paper, for a function $f$ locally lower semicontinuous at a stationary point $\bar{x}$, we obtain
complete characterizations of  the  Kurdyka-{\L}ojasiewicz (for short, K{\L}) property  and the exact estimate of the K{\L} modulus via the outer limiting subdifferential of an auxilliary function, and obtain a sufficient condition for verifying  sharpness of the K{\L} exponent.
By introducing a $\frac{1}{1-\theta}$-th subderivative $h$ for $f$ at $\bar{x}$, we show that the K{\L} property of $f$ at $\bar{x}$ with exponent $\theta\in [0, 1)$ can be inherited by   $h$ at $0$ with the same exponent $\theta$, and  that the K{\L} modulus of $f$ at $\bar{x}$ is bounded above by that of $(1-\theta)h$ at $0$. When $\theta=\frac12$, we obtain the reverse results under the strong  metrically subregularity of the subgradient mapping for the class of prox-regular, twice epi-differentiable and subdifferentially continuous functions by virtue of Moreau envelopes. We apply the obtained results to establish the K{\L} property with exponent $\frac12$ and to provide calculations of the K{\L} modulus for smooth functions, the pointwise max   of finitely many smooth functions and the $\ell_p$ ($0<p\leq 1$) regularized functions respectively. It is worth noting that these functions often appear in structured optimization problems.

\end{abstract}

{\bf Keywords}: Kurdyka-{\L}ojasiewicz property, outer limiting subdifferential, first (second) subderivative, graphical derivative, Moreau envelope, prox-regular and twice epi-differentiable function

{\bf AMS: } Primary, 65K10, 65K15; Secondary, 90C26, 49M37

\section{Introduction}
In connection with many important notions in variational analysis and optimization such as stability, error bounds and  the strong metric subregularity of subgradient mappings,  the property of quadratic growth has been studied intensively, see \cite{ag08,ag14,bllzd22,bs00, di15,dmn14,kpyz19,roc} and references therein. Recently this property  has been used in \cite{cln21,cln22,dl18,hlmy2021} for establishing the linear convergence of  optimization algorithms.
For an extended-real-valued function $f$ and a point $\bar{x}$ where 0 is a subgradient, it is well-known \cite[Theorem 13.24]{roc} that the quadratic growth condition holds at $\bar{x}$ if and only if $d^2f(\bar{x}\mid 0)(w)>0$ for all $w\not=0$, where $d^2f(\bar{x}\mid 0)$ is the second subderivative of $f$ at $\bar{x}$ for 0  and can be regarded as a  construction on primal space.

 For an extended-real-valued function $f$ which is prox-regular, twice epi-differenti- able and subdifferentially continuous at $\bar{x}$ for 0,  it has recently been shown in \cite[Theorem 3.7]{chnt21}, among other things,  that the quadratic growth condition holds at $\bar{x}$ if and only if $D(\partial f)(\bar{x}\mid 0)$ is positive definite in the sense that $\langle v, w\rangle>0$ for all $v\in D(\partial f)(\bar{x}\mid 0)(w)$ with $w\not=0$. Here  $D(\partial f)(\bar{x}\mid 0)$ is the subgradient graphical derivative  of $f$ at $\bar{x}$ for 0, which is defined as the graphical derivative acting on the subgradient mapping $\partial f$ and can be regarded as a  construction on dual  space. This result is a generalization of the corresponding one in \cite[Corollay 3.7]{ag14} presented for a convex function.

 The class of prox-regular functions covers all lower semicontinuous, proper, convex functions, lower-$\mathcal{C}^2$ functions, strongly  amenable functions, and fully amenable functions, hence a large core of functions of interest in variational analysis and optimization, see  \cite{pr96} and  \cite[Chapter 13]{roc}.  The  twice epi-differentiability    of a function was first introduced in \cite{roc85}, and then justified in \cite{roc89} for  fully
amenable functions  and in \cite{comi91, io91} for some composite
functions. In a recent paper \cite{ms20},  the twice epi-differentiablity of composite functions was justified under parabolic regularity \cite{mms22},   and some precise formulas for the second subderivatives under the metric subregularity constraint qualification were provided.
The  subdifferential continuity \cite[Definition 13.28]{roc}  of a function is often assumed for technical reasons, which guarantees that nearness of $(x, v)$ to $(\bar{x}, 0)$ in the graph of $\partial f$ automatically entails nearness of $f(x)$ to $f(\bar{x})$.

Besides the important quadratic growth property,  the   Kurdyka-{\L}ojasiewicz (for short, K{\L}) property is another  very important one studied  extensively in variational analysis and optimization.
For an extended-real-valued function $f$ and a point $\bar{x}$ where $f$ is finite and locally lower semicontinuous,  the   K{\L} property of $f$ at $\bar{x}$ amounts to the existence of some $\epsilon>0$ and  $\nu\in (0,+\infty]$ such that the K{\L}  inequality
\[
\psi'(f(x)-f(\bar{x}))\,d(0, \partial f(x))\geq 1
\]
holds for all $x$ with $\|x-\bar{x}\|\leq \epsilon$ and $f(\bar{x})<f(x)<f(\bar{x})+\nu$, where $d(0, \partial f(x))$ stands for the distance of 0 from the limiting subdifferential $\partial f(x)$ of $f$ at $x$,  and $\psi:[0, \nu)\to \R_+$, served as a desingularizing function, is assumed to be continuous, concave and of class $\mathcal{C}^1$ on $(0, \nu)$  with  $\psi(0)=0$ and  $\psi'>0$ over $(0, \nu)$. The term ``K{\L} property'' was first coined in Bolte et al. \cite{bdlm}.   The original and pioneering work of {\L}ojasiewicz \cite{lo63} and Kurdyka \cite{ku98} on differentiable functions  laid the foundation of the K{\L} property, which was later extended to nonsmooth functions by Bolte et al. \cite{bot2,bdls07}. Very recently, by employing nonsmooth desingularizing functions, a generalized version of the concave K{\L} property, together with its exact modulus, was introduced and studied in \cite{ww22}.

In this paper, we intend to study the K{\L} property associated with the  widely used  desingularizing function  $\psi(s):=\frac{1}{\mu}s^{1-\theta}$ for some exponent $\theta\in [0, 1)$ and modulus $\mu>0$, in which case the corresponding K{\L} inequality can be   written as
\begin{equation}\label{def-kl-theta-mu}
d(0, \partial f(x))\geq \frac{\mu}{1-\theta}(f(x)-f(\bar{x}))^\theta.
\end{equation}
The K{\L} property with its associated exponent (in particular the case of $\theta=\frac12$) and the  modulus plays a key role in estimating local convergence rates of many first-order algorithms, see \cite{ab,abrs,abs, fgp,lp2016, lp2018} and references therein.  To further elaborate the role of the K{\L} property in the convergence analysis, the following questions may arise for the quadruple $(f, \bar{x}, \theta, \mu)$:
\begin{description}
  \item[(Q1)] How to verify  the K{\L} property for $f$ at $\bar{x}$ with exponent $\theta$? 
    \item[(Q2)] Is it possible to argue the {\bf sharpness} of $\theta$  in the sense that $f$ satisfies the K{\L} property at $\bar{x}$ with exponent $\theta$ but not
    with any exponent smaller than $\theta$?
      \item[(Q3)]  What is the threshold for the modulus, i.e., the supremum of all the possible modulus $\mu>0$, along with some $\epsilon>0$ and $\nu\in (0, +\infty]$,  such that (\ref{def-kl-theta-mu}) holds? For the sake of simplicity, we denote such a threshold  by K{\L}$(f,\bar{x}, \theta)$.
    \item[(Q4)] Beyond the   class of $\mathcal{C}^2$ functions as justified in \cite{abrs,zlx16}, is there any larger class of functions such that once $f$ falls into it with the quadratic growth condition   being fulfilled at $\bar{x}$,    the K{\L} property of  $f$ then holds  automatically at $\bar{x}$ with exponent $\theta=\frac12$?
\end{description}

Part of the answers to the above  questions can be found in the literature.

For some answers to {\bf (Q1)},  we refer the reader to \cite{abrs,bnpb,lp2018,lmp2015,lws,wpb,ylp,zlx16} and references therein. In most of   these references,  the K{\L} exponent $\theta=\frac12$ was  verified for a convex function or  a nonconvex function  with a structure. By developing several important calculus rules for the K{\L} exponent,
    Li and Pong \cite{lp2018} recently estimated the K{\L} exponents for   functions  appearing in structured optimization problems.  This line of research was further explored  by Yu et al. \cite{ylp} to show that the K{\L} exponent can be preserved under the inf-projection operation, which is a significant generalization of the operation of taking the minimum of finitely many functions.
    However, for the purpose of verifying the K{\L} property for general nonconvex functions $f$ having no particular structure, there was no condition or framework proposed  by virtue of the theory of generalized differentiation, such as the second subderivative $d^2f(\bar{x}\mid 0)$ and the subgradient graphical derivative $D(\partial f)(\bar{x}\mid 0)$ just reviewed for the quadratic growth condition.

         As for {\bf (Q2)}, to our best of knowledge, there was no such a result in the literature discussing sharpness of the K{\L} exponent.

        Part of the answers to {\bf (Q3)} in the case of $\theta=0$ can be found in \cite{ca2014, ca2016, ers, lmy, lmy20}  in terms of the outer limiting subdifferential $\partial^>f(\bar{x})$ of $f$ at $\bar{x}$, which says that  K{\L}$(f,\bar{x},0)$ is bounded from above by the distance of 0 from a lower estimate  of $\partial^>f(\bar{x})$ .      As a subset of the  limiting subdifferential $\partial f(\bar{x})$ of $f$ at $\bar{x}$,   the outer limiting subdifferential $\partial^>f(\bar{x})$  consists  of limiting subgradients of $f$ at $\bar{x}$ that are generated from subgradients of $f$ at  nearby points $x$ of $\bar{x}$  with $f(x)$ close enough to $f(\bar{x})$ from above, see Definition \ref{def-subgradients} below.           The notion of outer limiting subdifferentials was  first introduced by Ioffe and Outrata in \cite{io2008}  to formulate a sufficient condition for error bounds. It was  recently shown in \cite[Theorem 4.1]{mlyy21} that   the level-set mapping $S:\alpha\mapsto \{x\mid f(x)\leq \alpha\}$ has the Lipschitz-like property relative to $\{\alpha\mid \alpha\geq f(\bar{x})\}$ at $f(\bar{x})$ for $\bar{x}$ if and only if 0 does not belong to the outer limiting subdifferential $\partial^>f(\bar{x})$. This demonstrates  the ability of  $\partial^>f(\bar{x})$ to characterize stability properties.

Our main contributions  and the organization of the paper are as follows.

In Section \ref{sec-notation}, we first recall some necessary notation and  mathematical preliminaries, and then for an extended-real-valued function $f$ and a point $\bar{x}$ where $f$ is finite and locally lower semicontinuous and by utilizing the outer limiting subdifferential $\partial^>g(\bar{x})$ with $g(x):=(\max\{f(x)-f(\bar{x}), 0\})^{1-\theta}$  and $\theta\in [0, 1)$, we give answers  to {\bf (Q1)-(Q3)} in an abstract level   as follows:
\begin{description}
  \item[(A1)]$f$ satisfies  the K{\L} property at $\bar{x}$ with exponent $\theta$ if and only if $0\not\in \partial^>g(\bar{x})$.
    \item[(A2)]If $0\not\in \partial^>g(\bar{x})$ with $\partial^>g(\bar{x})\not=\emptyset$, then   $\theta$ is a {\bf sharp}  K{\L} exponent for $f$ at $\bar{x}$.
      \item[(A3)] The  K{\L}  modulus of $f$ at $\bar{x}$ with exponent $\theta$  coincides with the distance of 0 from $\partial^>g(\bar{x})$, i.e., K{\L}$(f,\bar{x}, \theta)=d(0, \partial^>g(\bar{x}))$. 
\end{description}
See Theorem \ref{theo-relation} and its proof below for more details.  The equivalent description of the K{\L} property in {\bf (A1)-(A3)} can be regarded as a limit form of the  K{\L} property, and hence a framework for the study of the  K{\L} property.

In Section \ref{sec-lower-esti}, for a given extended-real-valued function $f$ and a point $\bar{x}$ where $f$ is finite and locally lower semicontinuous,  we first show in Lemma \ref{lem-lower-esti} that,  the outer limiting subdifferential $\partial^>f(\bar{x})$  of $f$ at $\bar{x}$ includes as a subset the outer limiting subdifferential $\partial^>[df(\bar{x})](0)$ of the subderivative $df(\bar{x})$ at 0. This inclusion  can be used to recover the lower estimates  of  $\partial^>f(\bar{x})$ presented in \cite{ca2014, ca2016, ers} for  functions with a particular structure.   With the help of Lemma \ref{lem-lower-esti}, we further show in Theorem \ref{theo-good-lower-bound} that, the K{\L} property of $f$ at $\bar{x}$ with   exponent $\theta\in [0, 1)$ can be inherited by its $\frac{1}{1-\theta}$-th subderivative $h$ (defined by (\ref{theta-subderitive})) at $0$ with the same exponent $\theta$. This is done by showing that the union set (\ref{tongyong-lower}) is a subset of $\partial^> g(\bar{x})$.  Moreover,  in Theorem \ref{theo-good-lower-bound}, we also show that K{\L}$(f,\bar{x}, \theta)$ is bounded above by K{\L}$((1-\theta)h, 0, \theta)$, and that if there is $w$ such that $0<h(w)<+\infty$, then $\partial^>g(\bar{x})\not=\emptyset$, as required in {\bf (A2)},  can be verified. Note that the $\frac{1}{1-\theta}$-th subderivative $h$ reduces  to  the (first) subderivative $df(\bar{x})$ when $\theta=0$  and to the second subderivative $d^2f(\bar{x}\mid 0)$ when $\theta=\frac12$.

  In Section \ref{sec-prox-twice-diff}, we focus on case of $\theta=\frac12$ .  If $f$ is prox-regular at $\bar{x}$ for 0 and $D(\partial f)(\bar{x}\mid 0)$ is nonsingular in the sense that $[D(\partial f)(\bar{x}\mid 0)]^{-1}(0)=\{0\}$, we show in Theorem \ref{theo-prox-regular-twice-epidiff}   that $f$ satisfies the K{\L} property at $\bar{x}$ with exponent $\frac12$, which is {\bf sharp} if there is some $w$ such that $0<d^2f(\bar{x}\mid 0)(w)<+\infty$. This is done by verifying $0\not\in \partial^>g(\bar{x})$ with   $g(x):=\sqrt{\max\{f(x)-f(\bar{x}), 0\}}$  and then relying on  {\bf (A1)} and {\bf (A2)}. In the circumstance that $f$ is not only prox-regular but also twice epi-differentiable and subdifferentially continuous at $\bar{x}$ for 0, we show in Theorem \ref{theo-prox-regular-twice-epidiff} {\bf (a)} that  the union set (\ref{tongyong-lower}) with $\theta=\frac12$ turns out to be equal to $\partial^> g(\bar{x})$ and that the modulus K{\L}$(f,\bar{x}, \frac12)$ coincides with its second subderivative counterpart K{\L}$(\frac12 d^2f(\bar{x}\mid 0), 0, \frac12)$. Similar results for the Moreau envelope functions $e_\lambda f$ for all $\lambda>0$ sufficiently small are presented in Theorem \ref{theo-prox-regular-twice-epidiff} {\bf (b)} and some connections between the K{\L} moduli of $f$ and  $e_\lambda f$  are established in  Theorem \ref{theo-prox-regular-twice-epidiff} {\bf (c)}.  As a direct consequence of Theorem \ref{theo-prox-regular-twice-epidiff}, we provide in Corollary \ref{cor-second-order-growth}  an answer to {\bf (Q4)} as follows:
\begin{description}
  \item[(A4)] Whenever $f$ falls into the class of prox-regular, twice epi-differentiable and subdifferentially continuous functions,
    the quadratic growth condition  at $\bar{x}$ guarantees  the K{\L} property of  $f$ at $\bar{x}$ with exponent $\theta=\frac12$.
\end{description}

In Section \ref{sec-basic-example},  we apply Theorems \ref{theo-relation} and \ref{theo-prox-regular-twice-epidiff}  to  establish the K{\L} property with exponent $\frac12$ and to provide   calculation of  K{\L}$(f,\bar{x}, \frac12)$ for  functions often appearing  in optimization problems:  smooth functions, the pointwise max of finitely many smooth functions,
   and the $\ell_p$ ($0<p\leq 1$) regularized functions.

\section{Notation and Mathematical Preliminaries} \label{sec-notation}
Throughout the paper we use the standard notation of variational analysis; see the seminal book \cite{roc} by Rockafellar and Wets.
The Euclidean norm of a vector $x$ is denoted by $||x||$, and the inner product of vectors $x$ and $y$ is denoted by $\langle x, y\rangle$.
   Let $\B$ denote the closed unit Euclidean ball and let $\Sph$ denote the unit sphere. We denote by $\B_\delta(x)$ the closed ball centered at $x$ with radius $\delta>0$.

 Let $A\subset \R^n$ be a nonempty set. We say that $A$ is locally closed at a point $x\in A$ if $A\cap U$ is closed for some closed neighborhood $U$ of $x$. We denote the interior,  the relative interior  and the closure  of $A$ respectively by $\inte A$,  $\ri A$ and $\cl A$.  
The distance from $x$ to $A$ is defined by
$
d(x,A):=\inf_{y\in A}||y-x||.
$
The indicator function $\delta_A$ of $A$ is defined by
\[
\delta_A(x):=\left\{
\begin{array}{ll}
0 &\mbox{if}\;x\in A,\\
+\infty &\mbox{otherwise}.
\end{array}
\right.
\]
Let $x\in A$. The notions of tangent cone and various normal cones are as follows.
\begin{description}
\item[(a)]   The  tangent  cone to $A$ at $x$ is denoted by  $T_A(x)$, i.e. $w\in T_A(x)$ if there exist sequences $t_k\downarrow 0$
and  $w_k\rightarrow w$ with  $x+ t_k w_k\in A$ for all $k$.
\item[(b)] The regular   normal cone, $\widehat{N}_A(x)$,  to $A$ at $x$ is the polar cone of $T_A(x)$. That is, $v\in \widehat{N}_A(x)$ if and only if $\langle v, w\rangle\leq 0$ for all $w\in T_A(x)$.
\item[(c)] A vector $v\in \R^n$ belongs to the   limiting  normal cone $N_A(x)$ to $A$ at $x$,
if there exist sequences $x_k\to x$ and $v_k\to v$ with $x_k\in A$ and $v_k\in \widehat{N}_A(x_k)$ for all $k$.
\end{description}

Let $f:\R^n\to \overline{\R}:=\R\cup\{\pm \infty\}$ be an extended real-valued function.  We denote the epigraph and the domain of $f$ by
\[
\epi f:=\{(x,\alpha)\mid f(x)\leq \alpha\}\quad\mbox{and}\quad \dom f:=\{x\mid f(x)<+\infty\},
\]
respectively.  We call $f$ a proper function if $f(x)<+\infty$ for at least one $x\in \R^n$, and $f(x)>-\infty$ for all $x\in \R^n$. $f$ is said to be lower semicontinuous (for short, lsc) if $\epi f$ is closed.

Let $\bar{x}$ be  a point with $f(\bar{x})$ finite.  $f$ is said to be locally lsc  at $\bar{x}$, if there is an $\epsilon > 0$ such that
all sets of the form $\{x\mid \|x-\bar{x}\|\leq \epsilon,\,f(x)\leq \alpha\}$ with $\alpha\leq f(\bar{x})+\epsilon$ are closed, see \cite[Definition 1.33]{roc}.  It is well-known that $f$ is locally lsc at $\bar{x}$ if and only if $\epi f$ is locally closed at $\left(\bar{x}, f(\bar{x})\right)$.

 We say that $f$ satisfies the quadratic growth condition at $\bar{x}$, if there is some $\kappa>0$ such that the inequality $f(x)\geq f(\bar{x})+\frac{\kappa}{2}\|x-\bar{x}\|^2$ holds  for all $x$ close enough to $\bar{x}$.

\begin{definition}\label{def-subgradients}
The notions of various subgradients  are defined as follows.
\begin{description}
\item[(a)] A vector $v\in \R^n$ is called a proximal subgradient of $f$ at $\bar{x}$, if there exist $\rho>0$ and $\delta>0$ such that
$
f(x)\geq f(\bar{x})+\langle v, x-\bar{x}\rangle -\frac12 \rho\|x-\bar{x}\|^2$ for all $x\in \B_\delta(\bar{x})$.
\item[(b)] A vector $v\in \R^n$ is a regular  subgradient of $f$ at $\bar{x}$, written $v\in \widehat{\partial} f(\bar{x})$, if
$f(x)\geq f(\bar{x})+\langle v, x-\bar{x}\rangle+o(||x-\bar{x}||)$. Equivalently, $v\in \widehat{\partial} f(\bar{x})$ if and only if
\[
\liminf_{x\to \bar{x}, \,x\not=\bar{x}}\frac{f(x)-f(\bar{x})-\langle v, x-\bar{x}\rangle}{\|x-\bar{x}\|}\geq 0.
\]
\item[(c)] A vector $v\in \R^n$ is a limiting subgradient of $f$ at $\bar{x}$, written $v\in \partial f(\bar{x})$, if there exist sequences $x_k\to \bar{x}$ and $v_k\to v$ with $f(x_k)\to f(\bar{x})$ and $v_k\in \widehat{\partial} f(x_k)$.
\item[(d)]  A vector $v\in \R^n$ is an outer limiting subgradient of $f$ at $\bar{x}$, written $v\in \partial^>f(\bar{x})$, if there exist sequences $x_k\to \bar{x}$ and $v_k\to v$ with $f(x_k)\to f(\bar{x})$, $f(x_k)>f(\bar{x})$ and $v_k\in  \partial  f(x_k)$.
\end{description}
\end{definition}
The outer limiting subdifferential set  $\partial^>f(\bar{x})$  of $f$ at $\bar{x}$ was first introduced in \cite{io2008}, see also \cite{knt10}.
   A closely related notion, called the right-sided subdifferential, was given in \cite{mor05} and defined by using a weak inequality $f(x_k)\geq f(\bar{x})$ instead of the strict inequality  $f(x_k)>f(\bar{x})$ used here. See  \cite[Definition 1.100 and Theorem 1.101]{mor} for more details on  the right-sided subdifferential and its applications.

For a set-valued mapping $S:\R^n\rightrightarrows \R^m$, we denote by
\[
\gph S:=\{(x, u)\mid u\in S(x)\}\quad\mbox{and}\quad \dom S:=\{x\mid S(x)\not=\emptyset\}
\]
  the graph   and the domain of $S$, respectively.
  The inverse mapping $S^{-1}:\R^m\rightrightarrows \R^n$ is defined by setting $S^{-1}(u):=\{x\mid u\in S(x)\}$.
  \begin{definition}
Consider a set-valued mapping $S:\R^n\rightrightarrows \R^m$ and a point $(\bar{x}, \bar{u})\in \gph S$.  The graphical  derivative of $S$ at $\bar{x}$ for $\bar{u}$ is
the mapping $DS(\bar{x}\mid \bar{u}):\R^n\rightrightarrows \R^m$ defined by
\[
z\in DS(\bar{x}\mid \bar{u})(w)\Longleftrightarrow
(w, z)\in T_{\gph S}(\bar{x}, \bar{u}).\]
The notation $DS(\bar{x}\mid \bar{u})$ is simplified to $DS(\bar{x})$ when $S$ is single-valued
  at $\bar{x}$, i.e., $S(\bar{x})=\bar{u}$.
\end{definition}

Following \cite[section 3H]{dr14}, we say that $S$ is metrically subregular at $\bar{x}$ for $\bar{u}\in S(\bar{x})$ with modulus $\kappa>0$ if
there exists a neighborhood $U$ of $\bar{x}$ such that $d(x, S^{-1}(\bar{u}))\leq \kappa d(\bar{u}, S(x))$ for all $x\in U$. The infimum of all
such $\kappa$ is the modulus of metrically subregularity, denoted by ${\rm subreg}\,S(\bar{x}\mid\bar{u})$. If in addition $\bar{x}$ is an isolated point of $S^{-1}(\bar{u})$, we say that $S$ is strongly metrically subregular  at $\bar{x}$ for $\bar{u}$. It is known from \cite[Theorem 4E.1]{dr14} that  $S$ is strongly metrically subregular  at $\bar{x}$ for $\bar{u}$ if and only if $DS(\bar{x}\mid \bar{u})^{-1}(0)=\{0\}$. Moreover, in the latter case, its modulus of (strong) metric subregularity is computed by
  \begin{equation}\label{def-smallest-norm}
{\rm subreg}\,S(\bar{x}\mid\bar{u})=1/\inf\{\|v\|\mid v\in DS(\bar{x}\mid \bar{u})(w),\,w\in \Sph\}.
  \end{equation}
As $DS(\bar{x}\mid \bar{u})^{-1}(0)=\{0\}$ means the only way to get $0\in DS(\bar{x}\mid \bar{u})(w)$ is to take $w=0$,  we say     that  the mapping $DS(\bar{x}\mid \bar{u})$ is nonsingular when $DS(\bar{x}\mid \bar{u})^{-1}(0)=\{0\}$. Following \cite{chnt21}, we say that the mapping $DS(\bar{x}\mid \bar{u})$ is positive definite if $\langle v, w\rangle>0$ for all $v\in DS(\bar{x}\mid \bar{u})(w)$ with $w\not=0$.

\begin{definition}[first and second subderivatives]\label{def-subd}
Consider a function $f:\R^n\to \overline{\R}$,  a point $\bar{x}\in \R^n$ with $f(\bar{x})$ finite and some $v\in \R^n$.
\begin{description}
\item[(a)]  The (first) subderivative of $f$ at $\bar{x}$ is  defined by
\[
df(\bar{x})(w):=\liminf_{t\downarrow 0, w'\to w}\frac{f(\bar{x}+tw')-f(\bar{x})}{t}.
\]
\item[(b)] The second subderivative of $f$ at $\bar{x}$ for   $v $ is   defined by
\[
d^2f(\bar{x}\mid v)(w):=\liminf_{t\downarrow 0, w'\to w}\frac{f(\bar{x}+tw')-f(\bar{x})-t\langle v, w'\rangle}{\frac12 t^2}.
\]
\item[(c)] $f$ is said to be twice epi-differentiable at $\bar{x}$ for $v$, if for every $w\in \R^n$ and choice of $t^\nu\downarrow 0$ there exists some $w^\nu\to w$ such that
\[
d^2f(\bar{x}\mid v)(w)=\lim_{\nu\to \infty}\frac{f(\bar{x}+t^\nu w^\nu)-f(\bar{x})-t^\nu\langle v, w^\nu\rangle}{\frac12 (t^\nu)^2}.
\]
\end{description}
\end{definition}

Following \cite{abrs,bot2,lp2018}, we introduce the notion of the  K{\L} property  associated with the  widely used desingularizing function  $\psi(s):=\frac{1}{\mu}s^{1-\theta}$ for some exponent $\theta\in [0, 1)$ and modulus $\mu>0$. For some later developments,  we also define the threshold for the modulus and the sharpness of the   exponent.
\begin{definition}\label{def-kl}
Consider a function $f:\R^n\to \overline{\R}$ and a point $\bar{x}$ where $f(\bar{x})$ is finite and locally lsc.
\begin{description}
    \item[(i)] We say that  $f$ satisfies the K{\L} property   at $\bar{x}$ with   exponent $\theta\in [0, 1)$ and  modulus $\mu>0$, if  there exist some  $\epsilon>0$ and  $\nu\in (0,+\infty]$ so that
\begin{equation}\label{def-kl-ori}
d(0, \partial f(x))\geq \frac{\mu}{1-\theta} \left(\max\{f(x)-f(\bar{x}), 0\}\right)^\theta
\end{equation}
whenever $\|x-\bar{x}\|\leq \epsilon$ and $f(x)<f(\bar{x})+\nu$.
\item[(ii)] Let $\theta\in [0, 1)$.  We define  by K{\L}$(f,\bar{x}, \theta)$ the supremum of all possible $\mu>0$, associated with some $\epsilon>0$ and $\nu\in (0, +\infty]$, such that (\ref{def-kl-ori}) holds for all $x\in \R^n$ with $\|x-\bar{x}\|\leq \epsilon$ and $f(x)<f(\bar{x})+\nu$, where the convention  $\sup\emptyset:=0$ is used.
   We call K{\L}$(f,\bar{x}, \theta)$ the  K{\L} modulus  of $f$ at $\bar{x}$ with exponent $\theta$.  Alternatively, we have \begin{equation}\label{kl-modulus-def}
 K{\L}(f,\bar{x}, \theta):=\liminf_{x\to_f \bar{x},\, f(x)>f(\bar{x})}\frac{(1-\theta)d(0, \partial f(x))}{(f(x)-f(\bar{x}))^\theta},
\end{equation}
where the lower limit is set to be $+\infty$  if there is no $x_k\to \bar{x}$ with $f(x_k)\to f(\bar{x})+$.
\item[(iii)]  If $f$ satisfies   the K{\L} property    at $\bar{x}$ with   exponent $\theta\in (0, 1)$
but not with any exponent  $\theta'\in [0, \theta)$, we say that $\theta$ is a {\bf sharp}
K{\L}  exponent of $f$ at $\bar{x}$. By convention we say that 0 is a {\bf sharp}   K{\L}
exponent of $f$ at $\bar{x}$,  if   $f$   satisfies   the K{\L} property    at $\bar{x}$ with exponent 0 and $\partial^>f(\bar{x})\not=\emptyset$.
\end{description}
\end{definition}

In terms of outer limiting subgradients, the limit form of the K{\L} property  and its characteristics on exponents and moduli  can be given as follows.
\begin{theorem}\label{theo-relation}  Consider a function $f:\R^n\rightarrow \overline{\R}$ and a point $\bar{x}\in \R^n$ where $f$ is finite and  locally lsc.   Let $0\leq\theta<1$ and let $g(x):=(\max\{f(x)-f(\bar{x}), 0\})^{1-\theta}$. Then we have
\begin{equation}\label{wcwf}
\partial^> g(\bar{x})=\limsup_{x\to _f\;  \bar{x}, f(x)>f(\bar{x})} \frac{(1-\theta)\partial f(x)}{(f(x)-f(\bar{x}))^\theta},
\end{equation}
and the following   hold:
\begin{description}
\item[(a)] $f$ satisfies the K{\L} property  at $\bar{x}$ with   exponent   $\theta$  if and only if
$
0\not\in \partial^>g(\bar{x})
$.
\item[(b)] If $\partial^> g(\bar{x})\not=\emptyset$, then
$0\in \partial^>\tilde{g}(\bar{x})$ holds for all $\tilde{g}(x):=(\max\{f(x)-f(\bar{x}), 0\})^{1-\tilde{\theta}}$ with $\tilde{\theta}<\theta$.  
\item[(c)]  $f$ satisfies the K{\L} property  at $\bar{x}$ with a {\bf sharp} exponent $\theta$ if
$0\not\in \partial^> g(\bar{x})$ with $\partial^> g(\bar{x})\not=\emptyset
$.
\item[(d)]
K{\L}$(f,\bar{x}, \theta)= d\left(0, \partial^> g(\bar{x})\right)$.
\end{description}
\end{theorem}
\begin{proof} As $f$ is assumed to be lsc, we get from \cite[Proposition 2.1]{yz}  that for all $x\in \R^n$ with $f(x)>f(\bar{x})$,
$
\partial  g(x)=\frac{(1-\theta)\partial f(x)}{(f(x)-f(\bar{x}))^\theta}
$, by virtue of which, the formula (\ref{wcwf}) follows immediately from the definition of the outer limiting subdifferential set.
By  (\ref{kl-modulus-def}) and (\ref{wcwf}), we get the equality in {\bf (d)}, from which the equivalence in {\bf (a)} follows in a straightforward way.
In view of the definition of  sharpness of the K{\L} exponent, we get {\bf (c)}  from {\bf (a)} and {\bf (b)} immediately. 
It remains to show {\bf (b)}. Let $v\in \partial^>g(\bar{x})$ and let $\tilde{\theta}<\theta$.   By definition  there exist $x^k\to_f \bar{x}$ with $f(x^k)>f(\bar{x})$ and  $v_k\in \partial f(x^k)$
such that $\frac{(1-\theta)v_k}{(f(x^k)-f(\bar{x}))^\theta}\to v$ implies 
\[
 \frac{(1-\theta)v_k}{(f(x^k)-f(\bar{x}))^{\tilde{\theta}}}= \frac{(1-\theta)v_k}{(f(x^k)-f(\bar{x}))^{\theta}}(f(x^k)-f(\bar{x}))^{\theta-\tilde{\theta}}\to 0,
\]
which implies by definition and \cite[Proposition 2.1]{yz}  that $0\in \partial^> \tilde{g}(\bar{x})$. This completes the proof. \end{proof}

\section{Preservation of the K{\L} property in passing to the subderivative}\label{sec-lower-esti}
In this section, we are given some  $\theta\in [0, 1)$,  a function $f:\R^n\rightarrow \overline{\R}$ and  a point  $\bar{x}$ where $f$ is finite and locally lsc. 
We will show in Theorem \ref{theo-good-lower-bound} that the K{\L} property of $f$ at $\bar{x}$ with exponent $\theta$ can be preserved by its $\frac{1}{1-\theta}$-th subderivative (cf. (\ref{theta-subderitive}))  at 0. This is done by first showing in Lemma \ref{lem-lower-esti} that,  the outer limiting subdifferential of $f$ at $\bar{x}$ includes as a subset the outer limiting subdifferential of its (first) subderivative $df(\bar{x})$ at 0, which is a result of independent interest.

\begin{lemma}\label{lem-lower-esti} Suppose that $f:\R^n\rightarrow \overline{\R}$ is finite and locally lsc at $\bar{x}$. In terms of $h:=df(\bar{x})$ and $W:=\{w\in \R^n\mid 0<h(w)<+\infty\}$, the following hold:
\begin{description}
\item[(a)] For each $w\in \R^n$ and  $v\in \widehat{\partial}h(w)$, there exist $t_l\to 0+$, $w_l\to w$ and  $v_l\to v$  such that  $v_l\in
\widehat{\partial}f(\bar{x}+t_lw_l)$ for all $l$  and
$h(w)=\lim_{l\to  \infty}\frac{f(\bar{x}+t_lw_l)-f(\bar{x})}{t_l}$.
\item[(b)] If $h$ is proper (i.e., h(0)=0),  then for  any $\bar{w}\in W$,
\begin{equation}\label{best_lower}
\partial h(\bar{w})\subset \partial^>h(0)\subset \partial^>f(\bar{x}),
\end{equation}
where $\partial^>h(0)=\displaystyle\cl\bigcup_{w\in W }\partial h(w)$ is nonempty if and only if  $W\not=\emptyset$.
\item[(c)] If $h$ is improper (i.e., $h(0)=-\infty$),   the results in {\bf (b)} hold  with $\partial^>h_+(0)$ in place of  $\partial^>h(0)$.
\end{description}
\end{lemma}
\begin{proof} We first show {\bf (a)}.   Let $w\in \R^n$ and let $v\in \widehat{\partial}h(w)$ be given arbitrarily. By the  definition of subderivative, there exist some   $\bar{t}_k\to 0+$ and $\bar{w}_k\to w$ such that
\begin{equation}\label{subderivative}
h(w)=\lim_{k\to \infty}\frac{f(\bar{x}+\bar{t}_k \bar{w}_k)-f(\bar{x})}{\bar{t}_k}.
\end{equation}
Let $\alpha_l\to 0+$ as $l\to +\infty$. For each $k$  and $l$, we consider the minimization of
\[
F_{k,l}(x):=f(x)-f(\bar{x})-\langle v, x-\bar{x}\rangle+ \frac{1}{\alpha_l \bar{t}_k}||x-(\bar{x}+ \bar{t}_k \bar{w}_k)||^2
\]
over the compact set $X:=\bar{x}+\bar{t}_k \bar{w}_k+ \bar{t}_k \alpha_l\B$. Since $F_{k,l}$ is locally lsc at $\bar{x}$ due to $f$ being assumed to be locally lsc at $\bar{x}$, the optimal solution set $\argmin_{x\in X}F_{k,l}(x)$ is nonempty. For each $k$ and $l$,  let $u_{k,l}\in \B$ be such that $\bar{x}+\bar{t}_k \bar{w}_k+ \bar{t}_k \alpha_l u_{k,l}\in  \argmin_{x\in X}F_{k,l}(x)$. Then we have
$F_{k,l}(\bar{x}+\bar{t}_k \bar{w}_k+\bar{t}_k \alpha_l u_{k,l})\leq  F_{k,l}(\bar{x}+\bar{t}_k \bar{w}_k)$, or explicitly
\begin{equation}\label{baoyu}
\frac{f(\bar{x}+\bar{t}_k \bar{w}_k+ \bar{t}_k \alpha_l u_{k,l})-f(\bar{x})}{\bar{t}_k}\leq \frac{f(\bar{x}+\bar{t}_k \bar{w}_k)-f(\bar{x})}{\bar{t}_k}+\langle v, \alpha_l u_{k,l}\rangle- \alpha_l||u_{k,l}||^2\quad \forall k, l.
\end{equation}
Since $u_{k,l}\in  \B$ for all $k$ and $l$,  by taking a subsequence if necessary, we may assume that for each $l$, there is some $u_l$ such that  $u_{k,l}\to u_l\in \B $ as $k\to +\infty$. Then by letting  $k\to \infty$, it follows  from  (\ref{subderivative}) and (\ref{baoyu}) that
$ 
\limsup_{k\to +\infty}\frac{f(\bar{x}+\bar{t}_k \bar{w}_k+ \bar{t}_k \alpha_l u_{k,l})-f(\bar{x})}{\bar{t}_k}\leq h(w)+ \langle v, \alpha_l u_l \rangle- \alpha_l ||u_l||^2$  for all $l$. By the definition of subderivative, we further have
$
h(w+\alpha_l u_l)\leq \liminf_{k\to +\infty}\frac{f(\bar{x}+\bar{t}_k \bar{w}_k+ \bar{t}_k \alpha_l u_{k,l})-f(\bar{x})}{\bar{t}_k}$  for all $l$. 
Therefore,  we get
$
h(w+\alpha_l u_l)\leq h(w)+ \langle v, \alpha_l u_l \rangle- \alpha_l ||u_l||^2$ for all $l$, or   alternatively  that,
\begin{equation}\label{bnmh}
\frac{h(w+\alpha_l u_l)-h(w)}{\alpha_l}\leq \langle v, u_l \rangle- ||u_l||^2\quad \forall l.
\end{equation}
Since   $u_l\in \B$ for all $l$,   by taking a subsequence if necessary, we may assume that $u_l\to u^*\in \B$ as $l\to \infty$. By letting $l\to \infty$, we get from (\ref{bnmh}) and the definition of subderivative that
\[
dh(w)(u^*)\leq \liminf_{l\to \infty}\frac{h(w+\alpha_l u_l)-h(w)}{\alpha_l}\leq \langle v, u^* \rangle- ||u^*||^2.
\]
As $v\in \widehat{\partial}h(w)$, we have $\langle v, u^* \rangle \leq dh(w)(u^*)$ and hence
$\langle v, u^* \rangle \leq \langle v, u^* \rangle- ||u^*||^2$.
This yields $u^*=0$.

For each $l$,   define $k(l)$ as a positive integer such that $k(l+1)>k(l)$ and $\|u_{k(l), l}-u_l\|\leq \frac{1}{l}$, which is possible because for each $l$, $u_{k,l}\to u_l$ as $k\to \infty$. Then we have $k(l)\to +\infty$  as $l\to \infty$. Due to $\|u_{k(l), l}\|\leq \|u_{k(l), l}-u_l\|+\|u_l\|\leq \frac{1}{l}+\|u_l\|$ and $u_l\to u^*=0$ as $l\to \infty$, we have $u_{k(l), l}\to 0$ as $l\to \infty$. By taking a subsequence if necessary, we assume that $u_{k(l), l}\in \inte \B$ for all $l$. For each $l$, according to the previous argument, $\bar{x}+\bar{t}_{k(l)}\bar{w}_{k(l)}+\bar{t}_{k(l)}\alpha_lu_{k(l), l}$ is a minimum of
$$
F_{k(l),l}(x):=f(x)-f(\bar{x})-\langle v, x-\bar{x}\rangle+ \frac{1}{\alpha_l \bar{t}_{k(l)}}||x-(\bar{x}+ \bar{t}_{k(l)} \bar{w}_{k(l)})||^2
$$
over the compact set $X:=\bar{x}+\bar{t}_{k(l)} \bar{w}_{k(l)}+ \bar{t}_{k(l)} \alpha_l\B$. As $u_{k(l), l}\in \inte \B$ for all $l$, we get from the optimality condition that  $0\in \widehat{\partial}F_{k(l),l}(\bar{x}+\bar{t}_{k(l)}\bar{w}_{k(l)}+\bar{t}_{k(l)}\alpha_lu_{k(l), l})$ or explicitly   $v-2u_{k(l), l}\in \widehat{\partial}f(\bar{x}+\bar{t}_{k(l)}\bar{w}_{k(l)}+\bar{t}_{k(l)}\alpha_lu_{k(l), l})$ for all $l$.
For each $l$, let
$$
t_l:=\bar{t}_{k(l)},\quad w_l:=\bar{w}_{k(l)}+\alpha_lu_{k(l), l}\quad\mbox{and}\quad v_l:=v-2u_{k(l), l}.
$$
 Then we have $t_l\to 0+$, $w_l\to w$ and $v_l\to v$ as $l\to \infty$, and moreover  $v_l\in \widehat{\partial}f(\bar{x}+t_lw_l)$ for all $l$.   Similarly as the derivation of (\ref{baoyu}), we get from the inequality $F_{k(l),l}(\bar{x}+t_lw_l)\leq F_{k(l),l}(\bar{x}+\bar{t}_{k(l)} \bar{w}_{k(l)})$ that
\begin{equation}\label{baoyu000}
\frac{f(\bar{x}+t_lw_l)-f(\bar{x})}{t_l}\leq \frac{f(\bar{x}+\bar{t}_{k(l)} \bar{w}_{k(l)})-f(\bar{x})}{\bar{t}_{k(l)}}+\langle v, \alpha_lu_{k(l), l}\rangle- \alpha_l||u_{k(l), l}||^2\quad \forall l.
\end{equation}
In view of (\ref{subderivative}) and the facts that $\alpha_l\to 0+$ and $u_{k(l), l}\to 0$ as $l\to \infty$, by letting $l\to \infty$ on both sides of (\ref{baoyu000}), we  get
$
\limsup_{l\to \infty}\frac{f(\bar{x}+t_lw_l)-f(\bar{x})}{t_l}\leq  h(w)$.
By the definition of subderivative, we have
$
h(w)\leq \liminf_{l\to \infty}\frac{f(\bar{x}+t_lw_l)-f(\bar{x})}{t_l}$
and hence the equality
$
h(w)= \lim_{l\to \infty}\frac{f(\bar{x}+t_lw_l)-f(\bar{x})}{t_l}$.
That is, the sequences $\{t_l\}$, $\{w_l\}$ and $\{v_l\}$ are the ones as required.

To show {\bf (b)},  by the positive homogeneity and the properness of $h$, we get by definition  the formula for $\partial^>h(0)$, which gives the first inclusion in (\ref{best_lower}).
Let $\bar{w}\in W $ and let $v\in \partial h(\bar{w})$.  By  definition   there are $w_k\to \bar{w}$ and $v_k\to v$ with $h(w_k)\to h(\bar{w})$ and $v_k\in \widehat{\partial} h(w_k)$. Clearly, we have $0<h(w_k)<+\infty$ for all $k$ sufficiently large. Then by {\bf (a)} and the definition of outer limiting subdifferential sets, we have $v_k\in \partial^>f(\bar{x})$ for all $k$ sufficiently large. As  $\partial^>f(\bar{x})$ is by definition closed, we  have $v\in  \partial^>f(\bar{x})$.
As $\bar{w}\in W $ and  $v\in \partial h(\bar{w})$ are given arbitrarily,  we get the second inclusion in  (\ref{best_lower}) immediately by taking the closedness of $\partial^>f(\bar{x})$ into account.   It remains to show that $\partial^>h(0)$ is nonempty whenever $W \not=\emptyset$.
For  $\bar{w}\in W$,   by the existence of subgradients \cite[Corollary 8.10]{roc} due to $h$ being lsc,  we assert that there exists a sequence $w^\nu\to  \bar{w}$ with $h(w^\nu)\to
  h(\bar{w})$ and  $\partial h(w^\nu)\not=\emptyset$ for all $\nu$.  Then for all $\nu$ large enough, we have $0<h(w^\nu) <+\infty$, i.e., $w^\nu\in W $. This entails the nonemptiness of  $\partial^>h(0)$.
  
To show {\bf (c)},  we consider the function $f_+(x):=\max\{f(x)-f(\bar{x}), 0\}$. Clearly, $f_+$ is locally lsc at $\bar{x}$ with $\partial^> f_+(\bar{x})=\partial^> f(\bar{x})$,  $df_+(\bar{x})=h_+$  and $\{w\mid 0<df_+(\bar{x})(w)<+\infty\}=W$. Moreover, $df_+(\bar{x})$ is a proper, lsc and positive homogeneous function with $\partial [df_+(\bar{x})](w)=\partial h(w)$ for all $w\in W$. So by applying {\bf (b)} to the function $f_+$,   the results in {\bf (b)} hold  with $\partial^>h_+(0)$ in place of  $\partial^>h(0)$.  The proof is completed. \end{proof}

\begin{remark}
    The idea of  using the auxiliary function $F_{k, l}(x)$ in our proof  for  perturbation analysis is borrowed from \cite[Lemma 1]{ers}. It is worth noting that the inclusions in (\ref{best_lower}) resemble the ones for subdifferential in \cite[Exercise 8.44]{roc}, that is, 
    $\partial h(\bar{w})\subset \partial h(0)\subset \partial f(\bar{x})$ with $\bar{w}\in \dom h$, where the requirement $h(\bar{w})>0$ in (\ref{best_lower}) shows a consistency with the requirement that the function values approach $f(\bar x)$ from above in the definition of the outer limiting subdifferential.
    \end{remark}

The following technical result  extends  the one  presented in \cite[Lemma 3.2]{chnt21} for a proper and positively homogeneous function of degree 2  to a positively homogeneous function of any degree $\gamma>0$, not necessarily being proper. Instead of using proximal subgradients for perturbation analysis  as  in the proof of \cite[Lemma 3.2]{chnt21}, we utilize regular subgradients in our proof.
\begin{lemma}\label{lem-qici}
Assume that $h:\R^n\to \overline{\R}$ is positively homogeneous of degree $\gamma>0$ in the sense that $h(\lambda w)=\lambda^\gamma h(w)$ for all $\lambda>0$ and $w\in \dom h$. Then for any $z\in \partial h(w)$ with $w\in \dom h$, we have $\langle z, w\rangle=\gamma h(w)$.
\end{lemma}
\begin{proof}  Let $z\in \partial h(w)$ with $w\in \dom h$. Then we have  $h(w)>-\infty$, for otherwise $\partial h(w)$ is undefined.
If $w=0$,  then by the positive  homogeneity of degree $\gamma$, we have  $h(w)=0$ and in this case $\langle z, w\rangle=\gamma h(w)$ holds trivially.
  Now assume that  $w\not=0$. By definition, there exist some $w_k\to w$ and $z_k\to z$  with $h(w_k)\to h(w)$  and  $z_k\in \widehat{\partial} h(w_k)$ for all $k$. By taking a subsequence if necessary, we assume that $w_k\not=0$ for all $k$. By the definition of  regular subgradients, we have
\begin{equation}\label{jdkankan}
\liminf_{u\to w_k,\;u\not=w_k}\frac{h(u)-h(w_k)-\langle z_k, u-w_k\rangle}{\|u-w_k\|}\geq 0 \quad \forall k.
\end{equation}
By letting $u=\lambda w_k$ with $\lambda\to 1+$ in  (\ref{jdkankan}), we get from the positive  homogeneity  of $h$ of degree $\gamma$  that
$
\gamma h(w_k)-\langle z_k, w_k\rangle \geq 0$ for all $k$.
Similarly, by letting $u=\lambda w_k$ with $\lambda\to 1-$ in  (\ref{jdkankan}), we  get
$
-\gamma h(w_k)+\langle z_k, w_k\rangle \geq 0$  for all $k$.
Therefore, we have
$
 \langle z_k, w_k\rangle = \gamma h(w_k)$ for all $k$, which implies   $\langle z, w\rangle=\gamma h(w)$ by letting $k\to \infty$ on both sides. This completes the proof. \end{proof}

By a direct application of Lemma \ref{lem-lower-esti} and  the  limit form of the  K{\L} property   refined in Theorem \ref{theo-relation}, our main results in this section are summarized in the following.
\begin{theorem}\label{theo-good-lower-bound}
Suppose that $f:\R^n\rightarrow \overline{\R}$ is finite and locally lsc at $\bar{x}$. Let $0\leq\theta<1$ and $g(x):=(\max\{f(x)-f(\bar{x}), 0\})^{1-\theta}$.  Define   $h:\R^n\to \overline{\R}$ by
\begin{equation}\label{theta-subderitive}
h(w):=\liminf_{\tau\to 0+,\,w'\to w}\frac{f(\bar{x}+\tau w')-f(\bar{x})}{(1-\theta) \tau^{\frac{1}{1-\theta}}},
\end{equation}
and let
\begin{equation}\label{xiaoyang}
W:=\{w\in \R^n\mid 0<h(w)<+\infty\}.
\end{equation}
The following hold:
\begin{description}
\item[(a)] The function $h$ is lsc   and positively homogeneous of degree $\frac{1}{1-\theta}$ with
$h(w)=\langle v, w\rangle$ for all $v\in \partial [(1-\theta)h](w)$. Moreover, we  have either $h(0)=0$ or $h(0)=-\infty$, and  $h$ is proper if and only if $h(0)=0$.
\item[(b)] The outer limiting subdifferential set  $\partial^>g(\bar{x})$ has
\begin{equation}\label{tongyong-lower}
\displaystyle \bigcup_{w\in W  } \frac{(1-\theta)\partial [ (1-\theta) h](w) }{[ (1-\theta) h(w)]^{\theta}}
\end{equation}
as a subset (possibly not closed), which   consists of nonzero vectors,   and  is nonempty if and only if  $W\not=\emptyset$. Moreover, we have  
$\partial^>[dg(\bar{x})](0)=\cl \displaystyle \bigcup_{w\in W  } \frac{(1-\theta)\partial [ (1-\theta) h](w) }{[ (1-\theta) h(w)]^{\theta}}$. 
\item[(c)] If  $W\not=\emptyset$,  then $f$ cannot satisfy the K{\L} property  at $\bar{x}$ with any exponent $\theta'\in [0, \theta)$.
   \item[(d)] Let $\mu>0$. If $f$ satisfies the K{\L} property  at $\bar{x}$ with exponent $\theta$ and   modulus $\mu$, then
\begin{equation}\label{haojinglian}
d\left(0, \partial[(1-\theta) h](w)\right)\geq \frac{\mu}{1-\theta}[(1-\theta)h(w)]^{\theta}\quad \forall w\in W.
\end{equation}
\item[(e)] If $h$ is proper, we have  K{\L}$(f,\bar{x}, \theta)\leq K{\L}((1-\theta)h, 0, \theta)$. Otherwise,  K{\L}$(f,\bar{x}, \theta)\leq K{\L}((1-\theta)h_+, 0, \theta)$.
\end{description}
\end{theorem}
\begin{proof} For all $\tau>0$, define the $\frac{1}{1-\theta}$-th-order difference quotients  $\Delta_\tau f(\bar{x}):\R^n\to \overline{\R}$ by
$\Delta_\tau f(\bar{x})(w):=\frac{f(\bar{x}+\tau w)-f(\bar{x})}{(1-\theta) \tau^{\frac{1}{1-\theta}}}$. 
Then the function $h$ is given by the lower epi-limit of the family of  functions $\Delta_\tau f(\bar{x})$ as $\tau\to 0+$, and   thus lsc (cf. \cite[Proposition 7.4 (a)]{roc}). The positive homogeneity of degree $\frac{1}{1-\theta}$   arises from the form of the $\frac{1}{1-\theta}$-th-order difference quotients in having
$\Delta_\tau f(\bar{x})(\lambda w)=\lambda^{\frac{1}{1-\theta}}\Delta_{\tau\lambda} f(\bar{x})(w)$ for $\lambda>0$. By the lower semicontinuity and  positive homogeneity of degree $\frac{1}{1-\theta}$ of $h$, one can easily verify  that either $h(0)=0$ or $h(0)=-\infty$, and that  $h$ is proper if and only if $h(0)=0$. Moreover, by Lemma \ref{lem-qici}, we get $h(w)=\langle v, w\rangle$  for all $w\in \R^n$ and $v\in \partial [(1-\theta)h](w)$, implying that the subset defined  by (\ref{tongyong-lower}) does not contain  0.    Clearly, $g$ is by definition  locally lsc at $\bar{x}$ with $g(\bar{x})=0$.  Moreover,  we  have by definition
\begin{equation}\label{jiandanfaze}
dg(\bar{x}) =[(1-\theta)h ]_+^{1-\theta} \quad\mbox{and}\quad  \partial [dg(\bar{x})](w)=\frac{(1-\theta)\partial [ (1-\theta) h](w) }{[ (1-\theta) h(w)]^{\theta}}\quad \forall w\in W.
\end{equation}
Therefore, we have $dg(\bar{x})(0)=0$, and  $0<dg(\bar{x})(w)<+\infty$ if and only if $w\in W$.   Then by applying Lemma \ref{lem-lower-esti} {\bf (b)} to $g$, we get all the results in {\bf (a)} and {\bf (b)}.

 If $W\not=\emptyset$, we have $\partial^> g(\bar{x})\not=\emptyset$. It then follows from Theorem \ref{theo-relation} {\bf (a)} and {\bf (b)} that   $f$ cannot satisfy the K{\L} property  at $\bar{x}$ with any exponent $\theta'\in [0, \theta)$. This verifies  {\bf (c)}. According to {\bf (b)} and (\ref{jiandanfaze}), we have  $\partial^>g(\bar{x})\supset \partial^>[dg(\bar{x})](0)= \partial^>[(1-\theta)h]_+^{1-\theta}(0)$, which implies by Theorem \ref{theo-relation} {\bf (d)} that   K{\L}$(f,\bar{x}, \theta)\leq$ K{\L}$((1-\theta)h_+, 0, \theta)$. That is,  if $f$ satisfies the K{\L} property  at $\bar{x}$ with exponent $\theta$ and  modulus $\mu$, then there exist some $\epsilon>0$ and $\nu\in (0, +\infty)$ so that
$ d(0, \partial[(1-\theta)h]_+(w))\geq \frac{\mu}{1-\theta}((1-\theta)h(w))_+^\theta$
whenever $\|w\|\leq \epsilon$ and $h(w)<\nu$. This latter condition can be refined as (\ref{haojinglian}) by virtue of the lower semicontinuity  and  the positive homogeneity of degree $\frac{1}{1-\theta}$ of $h$.   If $h$ is proper, we have $h(0)=0$ and  by definition  K{\L}$((1-\theta)h_+, 0, \theta)=$  K{\L}$((1-\theta)h, 0, \theta)$. This completes the proof. \end{proof}

\begin{remark}
We refer the reader to \cite[Example 2.1]{lmy} for situations where the subset   (\ref{tongyong-lower}) of  $\partial^> g(\bar{x})$  is not closed in the case of $\theta=0$. The lower estimates given in  \cite[Theorem 1]{ers} for the outer limiting subdifferential of a pointwise min  of  finitely many functions having  sublinear Hadamard directional derivatives, and in \cite[Corollary 3]{ers}  and  \cite[Theorem 3.2]{ca2016} for  the outer subdifferential of  the pointwise max of  finitely many smooth functions,  can be easily recovered from  (\ref{tongyong-lower}) with $\theta=0$.

According Theorem \ref{theo-good-lower-bound} {\bf (c)}, once the K{\L} property of $f$ at $\bar{x}$ with some exponent $\theta\in [0, 1)$ is verified,  it suffices to show  $W\not=\emptyset$  in order to justify the sharpness of $\theta$.

Note that the function $h$ given by (\ref{theta-subderitive}) reduces to  the (first) subderivative $df(\bar{x})$  of $f$ at $\bar{x}$ in the case of $\theta=0$, and to  the second subderivative $d^2f(\bar{x}\mid 0)$  of $f$ at $\bar{x}$  for 0 in the case of $\theta=\frac12$. Therefore,  the function $h$ given by (\ref{theta-subderitive}) in the general case can be regarded as the  $\frac{1}{1-\theta}$-th subderivative of $f$ at $\bar{x}$. 
This  subderivative function, as demonstrated  by Theorem \ref{theo-good-lower-bound} {\bf (d)} and {\bf (e)},  can  preserve the K{\L} property from that of $f$ with the same exponent, but not conversely. Below is a nontrivial example.
\end{remark}

For an illustration of Theorem \ref{theo-good-lower-bound} with $\theta=\frac12$, we consider a function $f:\R^n\to \overline{\R}$ and a point $\bar{x}$ around which $f$ is of class $\mathcal{C}^2$ with $\nabla f(\bar{x})=0$.
In this case, $d^2f(\bar{x}\mid 0)=\langle \nabla^2f(\bar{x})w, w\rangle$ is a proper quadratic function with    $\partial [\frac12 d^2f(\bar{x}\mid 0)](w)=\{\nabla^2f(\bar{x}) w\}$.
By the spectral decomposition theorem and  (\ref{kl-modulus-def}),  we   get
\[
 \mbox{K}{\L} (\frac12 d^2f(\bar{x}\mid 0), 0, \frac12 )=\min\left\{\frac{\frac12\|\nabla^2f(\bar{x})w\|}{\sqrt{\frac12\langle w, \nabla^2f(\bar{x})w\rangle}}\mid \langle w, \nabla^2f(\bar{x})w\rangle>0\right\}=\sqrt{\frac{\lambda_{\min}}{2}}>0
\]
where  $\lambda_{\min}$ is the smallest positive eigenvalue of $\nabla^2f(\bar{x})$ if $\nabla^2f(\bar{x})$ is not negative-semidefinite, and otherwise it is
 set to be  $+\infty$.

However, the inequality K{\L}$(f,\bar{x}, \frac12)\leq$ K{\L}$(\frac12 d^2f(\bar{x}\mid 0), 0, \frac12)$ in Theorem \ref{theo-good-lower-bound} {\bf (e)} may hold as a strict one, as can be  seen from the instance that
$f_0(x):=-(x_1+x_2-6)^2-16\sqrt{|x_1|}-16\sqrt{|x_2|}$
with $\bar{x}:=(1,1)^T$. Clearly, $f_0$ is of class $\mathcal{C}^\infty$ around $\bar{x}$  with $\nabla f_0(\bar{x})=0$  and a singular Hessian matrix
$
\nabla^2 f_0(\bar{x})=\left[
\begin{array}{cc}
2 & -2 \\
-2 & 2
\end{array}
\right]$  having two eigenvalues: $4$ and 0.
By the third-order Taylor expansion, we have
$f_0(\bar{x}+\tau w)-f_0(\bar{x})=(w_1-w_2)^2\tau^2+\tau^3(-w_1^3-w_2^3) +o(\tau^3)$ for all $w\in \R^2$. Then the third subderivative  $h$ defined by (\ref{theta-subderitive}) with $\theta=\frac23$ can be calculated to be
\[
h(w)=\left\{
\begin{array}{ll}
-3(w_1^3+w_2^3) & \mbox{if}\;w_1=w_2,\\
+\infty &\mbox{otherwise},
\end{array}
\right.
\]
and accordingly $W$ defined by (\ref{xiaoyang}) is exactly the nonempty set $\{t(-1, -1)^T\mid t>0\}$. By Theorem  \ref{theo-good-lower-bound} {\bf (c)},  $f_0$ cannot satisfy the K{\L} property  at $\bar{x}$ with any exponent $\theta'\in [0, \frac23)$. Therefore,
we  have 
\[
\mbox{K}{\L}(f_0, \bar{x}, \frac12)=0 \,<\,\mbox{K}{\L}(\frac12 d^2f_0(\bar{x}\mid 0), 0, \frac12)=\sqrt{\frac{\lambda_{\min}}{2}}=\sqrt{2}.
\]

This unsatisfactory situation can be avoided if the nonsingularity of  $\nabla^2f(\bar{x})$ is presented,
as can be seen from Proposition \ref{prop-nonsingular} below, where we  show  by assuming the nonsingularity of  $\nabla^2f(\bar{x})$   that,
 the inequality $\mbox{K}{\L}(f,\bar{x}, \frac12)\leq \mbox{K}{\L}(\frac12 d^2f(\bar{x}\mid 0), 0, \frac12)$ holds as an equality  and the subset (\ref{tongyong-lower}) is closed and coincides with the outer limiting subdifferential set  $\partial^> g(\bar{x})$ with $\theta=\frac12$. 
 These  results are obtained as a direct consequence of some more general results presented for prox-regular and twice epi-differentiable functions in the next section.

\section{The K{\L} property  for a prox-regular and twice epi-differentiable function and its Moreau envelopes}\label{sec-prox-twice-diff}
In this section, we study the K{\L} property with exponent $\frac12$  for  a prox-regular and twice epi-differentiable function by virtue of its Moreau envelopes.  

To begin with,  we  make it clear that the function $g$ defined previously for general $\theta\in [0, 1)$ is now specified  with $\theta=\frac12$ in this section:
 \begin{equation}\label{def-gg}
      g(x):=\sqrt{\max\{f(x)-f(\bar{x}), 0\}},
 \end{equation} 
 and  recall the notions of prox-regularity and subdifferential continuity as follows.
 \begin{definition}\cite[Definition 13.27]{roc}\label{def-prox-reg} A function $f:\R^n\to \overline{\R}$ is prox-regular at $\bar{x}$ for $\bar{v}$ if $f$ is finite and locally lsc at $\bar{x}$ with $\bar{v}\in \partial f(\bar{x})$, and there exist $\epsilon>0$ and $\rho\geq 0$ such that
$$
f(x')\geq f(x)+\langle v, x'-x\rangle-\frac{\rho}{2}\|x'-x\|^2\quad \forall x'\in \B_\epsilon(\bar{x})
$$
 when $v\in \partial f(x)$, $\|v-\bar{v}\|<\epsilon$, $\|x-\bar{x}\|<\epsilon$ and $f(x)<f(\bar{x})+\epsilon$. When this holds for all $\bar{v}\in \partial f(\bar{x})$, $f$ is said to be prox-regular at $\bar{x}$.
\end{definition}
\begin{definition}\cite[Definition 13.28]{roc}\label{def-sc}
A function $f:\R^n\to \overline{\R}$ is called subdifferentially continuous at $\bar{x}$ for $\bar{v}$ if $\bar{v}\in \partial f(\bar{x})$ and, whenever $(x^\nu, v^\nu)\to (\bar{x}, \bar{v})$ with $v^\nu\in \partial f(x^\nu)$, one has $f(x^\nu)\to f(\bar{x})$. If this holds for all $\bar{v}\in \partial f(\bar{x})$, $f$ is said to be subdifferentially continuous at $\bar{x}$.
\end{definition}
As pointed out in the Introduction, the class of prox-regular functions covers a large core of functions of interest in variational analysis and optimization  \cite{pr96, roc},  and the  subdifferential continuity  of a function is often assumed for technical reasons.

 For a function $f:\R^n\to \overline{\R}$  that is  prox-regular  at $\bar{x}$ for $0$, we will show that,  every outer
 limiting subgradient $u\in \partial^>g(\bar{x})$   corresponds to a nonnegative scalar $\tau$ and a vector $w\in \Sph$ such that
\begin{equation}\label{tfdf}
\sqrt{2\tau} u\in D(\partial f)(\bar{x}\mid 0)(w),
\end{equation}
where $\tau$,  $w$, and $\sqrt{2\tau} u$ can be generated from nearby points $(x, v)$ of $(\bar{x}, 0)$ in $\gph \partial f$ with $f(x)$ close enough to $f(\bar{x})$ from above.

\begin{lemma}\label{lem-prox-regular} Assume that $f:\R^n\to \overline{\R}$ is prox-regular  at $\bar{x}$ for $0$. Let $g$ be given by (\ref{def-gg}).  For every  $u\in \partial^>g(\bar{x})$,    there exist some $\tau\geq 0$,  $w\in \Sph$,   $x^k\to_f \bar{x}$ with $f(x^k)>f(\bar{x})$ and $v^k\in \partial f(x^k)$  such that
$$
\frac{f(x^k)-f(\bar{x})}{\frac{1}{2}\|x^k-\bar{x}\|^2}\to \tau,\quad\quad  \frac{x^k-\bar{x}}{\|x^k-\bar{x}\|}\to w  \quad\mbox{and}\quad  \frac{v^k}{\|x^k-\bar{x}\|}\to  \sqrt{2\tau} u\in D(\partial f)(\bar{x}\mid 0)(w).
$$
\end{lemma}
\begin{proof} Let $u\in \partial^>g(\bar{x})$.    By definition there are $x^k\to_f \bar{x}$ with $f(x^k)>f(\bar{x})$ and $v^k\in \partial f(x^k)$ such that
\begin{equation}\label{bsp}
v^k/\sqrt{4(f(x^k)-f(\bar{x}))}=\frac{v^k}{\|x^k-\bar{x}\|}/\sqrt{2\frac{f(x^k)-f(\bar{x})}{\frac12\|x^k-\bar{x}\|^2}} \to u.
\end{equation}
 As $f(x^k)\to f(\bar{x})$, we have $v^k\to 0$. In view of the prox-regularity of $f$ at $\bar{x}$ for $0$,  we can find some $\rho\geq 0$  such that $f(\bar{x})\geq f(x^k)+\langle v^k, \bar{x}-x^k\rangle-\frac{\rho}{2}\|x^k-\bar{x}\|^2$ for all $k$ large enough, implying that
\begin{equation}\label{gj-suofang}
 \frac{\|v^k\|}{\|x^k-\bar{x}\|}+\frac{\rho}{2}\geq \langle \frac{v^k}{\|x^k-\bar{x}\|}, \frac{x^k-\bar{x}}{\|x^k-\bar{x}\|}\rangle +\frac{\rho}{2}\geq \frac{f(x^k)-f(\bar{x})}{\|x^k-\bar{x}\|^2}>0\quad \forall k.
\end{equation}
Combing (\ref{bsp}) and (\ref{gj-suofang}), we have $
\frac{\|v^k\|}{\|x^k-\bar{x}\|}/\sqrt{ \frac{4\|v^k\|}{\|x^k-\bar{x}\|}+2\rho}\leq \|u\|+1
$ for all $k$ sufficiently large,
 which  implies the boundedness of the sequence $\{\frac{v^k}{\|x^k-\bar{x}\|}\}$ and hence that of the sequence  $\{ \frac{f(x^k)-f(\bar{x})}{\|x^k-\bar{x}\|^2}\}$. By taking some subsequences if necessary, we assume that there exist some  $w\in \Sph$, $v^*\in \R^n$   and   $\tau\geq 0$ such that   $\frac{x^k-\bar{x}}{\|x^k-\bar{x}\|}\to w$,   $\frac{v^k}{\|x^k-\bar{x}\|}\to v^*$ and $\frac{f(x^k)-f(\bar{x})}{\frac{1}{2}\|x^k-\bar{x}\|^2}\to \tau$.
      As $v^k\in \partial f(x^k)$ for all $k$ with  $x^k\equiv\bar{x}+\|x^k-\bar{x}\|\,\frac{x^k-\bar{x}}{\|x^k-\bar{x}\|}$ and $v^k\equiv 0+\|x^k-\bar{x}\|\,\frac{v^k}{\|x^k-\bar{x}\|}$,  we  have by definition  $v^*\in D(\partial f)(\bar{x}\mid 0)(w)$. It then follows from (\ref{bsp})
that  $v^*=\sqrt{2\tau}u$.  This completes the proof. \end{proof}

Whenever  $f$ is not only prox-regular and subdifferentially continuous, but also twice epi-differentiable (cf. Definition   \ref{def-subd}) at $\bar{x}$ for $0$ ,  we summarize in Lemma \ref{lem-prox-twice-epidiff} below  a few nice properties on  $f$  and the Moreau envelopes $e_\lambda f$ and $e_\lambda[\frac{1}{2}d^2f(\bar{x}\mid 0)]$, most of which can be found in  \cite{pr96} and  \cite[Chapter 13]{roc}.

 For a proper, lsc function $f:\R^n\to \overline{\R}$ and parameter value $\lambda>0$, the Moreau envelope function $e_\lambda f$ and proximal mapping $P_\lambda f$ are defined by
\[
e_\lambda f(x):=\inf_w\{f(w)+\frac{1}{2\lambda}\|w-x\|^2\}\quad \mbox{and}\quad P_\lambda f(x):=\arg\min_w\{f(w)+\frac{1}{2\lambda}\|w-x\|^2\}.
\]
Such a function is prox-bounded if there exists $\lambda>0$ such that $e_\lambda f(x)>-\infty$ for some $x\in \R^n$. In particular, if $f$ is bounded from below, then $f$ is prox-bounded.

\begin{lemma}\label{lem-prox-twice-epidiff}
Assume that  $f:\R^n\to \overline{\R}$ is prox-regular, subdifferentially continuous  and twice epi-differentiable at $\bar{x}$ for $0$. The following properties  hold:
\begin{description}
\item[(a)]  The function $d^2f(\bar{x}\mid 0)$ is proper, lsc,  prox-bounded and positively homogeneous of degree 2. Moreover, it is  prox-regular and subdifferentially continuous everywhere  with
\begin{equation}\label{jbcz}
\partial [\frac{1}{2} d^2f(\bar{x}\mid 0)]=D(\partial f)(\bar{x}\mid 0)
\end{equation}
and
\begin{equation}\label{wjfdyjbdl}
d^2f(\bar{x}\mid 0)(w)=\left\{
\begin{array}{ll}
\langle v, w\rangle &  \forall v\in D(\partial f)(\bar{x}\mid 0)(w),\\
+\infty &\forall w\not\in \dom D(\partial f)(\bar{x}\mid 0).
\end{array}
\right.
\end{equation}
\item[(b)]  For all $\lambda>0$ sufficiently small  there is a neighborhood of $\bar{x}$ on which $P_\lambda f$ is   single-valued and Lipschitz continuous with $P_\lambda f(\bar{x})=\bar{x}$ and $P_\lambda f=(I+\lambda \partial f)^{-1}$, and   $e_\lambda f$ is of class $\mathcal{C}^{1+}$ with $\nabla e_\lambda f(\bar{x})=0$ and $$\nabla e_\lambda f=\lambda^{-1}[I-P_\lambda f]=(\lambda I+(\partial f)^{-1})^{-1}.$$
  Moreover,  for all $\lambda>0$ sufficiently small, the Moreau envelope function $e_\lambda f$ is prox-regular,  subdifferentially continuous,  twice epi-differentiable (and twice semidifferentiable, cf. \cite[Definition 13.6]{roc}) at $\bar{x}$  and the  Moreau envelope function   $e_\lambda[\frac12 d^2f(\bar{x}\mid 0)]$ is of class $\mathcal{C}^{1+}$ and  positively homogeneous of degree 2,  with
\begin{equation}\label{fjjy000}
e_\lambda [\frac12 d^2f(\bar{x}\mid 0)]=  d^2[\frac{1}{2}e_\lambda f(\bar{x})],
\end{equation}
\begin{equation}\label{jbcz000}
\nabla e_\lambda [\frac12 d^2f(\bar{x}\mid 0)]=D(\nabla e_\lambda f)(\bar{x})=\left[\,\lambda I+ [D(\partial f)(\bar{x}\mid 0)]^{-1}\,\right]^{-1},
\end{equation}
and
\begin{equation}\label{jbcz111}
e_\lambda (  d^2f(\bar{x}\mid 0))(w)=\langle \nabla e_\lambda (\frac12 d^2f(\bar{x}\mid 0))(w), w\rangle\quad \forall w\in \R^n.
\end{equation}
\end{description}
\end{lemma}
\begin{proof}  By  \cite[Proposition 13.5]{roc}, $d^2f(\bar{x}\mid 0)$ is lsc and  positively homogeneous of degree 2.   In view of the fact that 0 is a proximal subgradient of $f$ at $\bar{x}$, we get from \cite[Theorem 4.1]{mms22} that    $d^2f(\bar{x}\mid 0)$ is proper  and prox-bounded.  In view of the fact that $f$ is assumed to be  prox-regular, subdifferentially continuous  and twice epi-differentiable at $\bar{x}$ for $0$, we get the equality (\ref{jbcz}) directly  from \cite[Theorem 13.40]{roc}.  The subdifferential  continuity of $d^2f(\bar{x}\mid 0)$ and the equality (\ref{wjfdyjbdl}) follow directly from Lemma \ref{lem-qici} and the equality (\ref{jbcz}). The prox-regularity of  $d^2f(\bar{x}\mid 0)$ follows from \cite[Proposition 13.49]{roc}. That is, all the results in
{\bf (a)} follow.  Observing that all the properties involved depend on the nature of $f$ in some neighborhood of $\bar{x}$, there's no harm, therefore, by adding to $f$ the indicator of some neighborhood of $\bar{x}$ so as to make $f$ being bounded from below and hence prox-bounded. Thus, without  assuming explicitly that $f$
 is prox-bounded, we can get all the results in {\bf (b)} from \cite[Proposition 13.37 and Exercise 13.45]{roc} in a straightforward way. This completes the proof. \end{proof}

For a vector $w\in \R^n$ and  a function $f:\R^n\to \overline{\R}$ which is  prox-regular, subdifferentially continuous  and twice epi-differentiable at some $\bar{x}$ for $0$,  the equality
\begin{equation}\label{xianzhi-semi}
d^2f(\bar{x}\mid 0)(w)=\lim_{k\to \infty}\frac{f(\bar{x}+t_kw_k)-f(\bar{x})}{\frac12 t_k^2}
\end{equation}
is not always guaranteed for any sequences $t_k\to 0+$ and $w_k\to w$ (cf. \cite[Example 13.10]{roc}), unless  there is  a convergent sequence $v_k$ such that $t_kv_k\in \partial f(\bar{x}+t_kw_k)$ for all $k$ sufficiently large, as can be seen from below.
\begin{lemma}\label{lem-rest-semi}
Assume that  $f:\R^n\to \overline{\R}$ is prox-regular, subdifferentially continuous  and twice epi-differentiable at $\bar{x}$ for $0$.
For any sequences $t_k\to 0+$ and  $w_k\to w$, (\ref{xianzhi-semi}) holds if there is a convergent  sequence $v_k$ such that $t_kv_k\in \partial f(\bar{x}+t_kw_k)$ for all $k$ sufficiently large.
\end{lemma}
\begin{proof} Let  $t_k\to 0+$ and $w_k\to w$.  Assume that there are some $v\in \R^n$ and
 $v_k\to v$ with $t_kv_k\in \partial f(\bar{x}+t_kw_k)$ for all $k$ sufficiently large.  Let $\lambda>0$ be sufficiently small such that all the properties
 in Lemma \ref{lem-prox-twice-epidiff} {\bf (b)} hold, and let $y_k:=\bar{x}+t_kw_k+\lambda t_kv_k$. Then we have $\bar{x}+t_kw_k\in (I+\lambda \partial f)^{-1}(y_k)$ for all $k$ sufficiently large.
By Lemma \ref{lem-prox-twice-epidiff} {\bf (b)}, we have $\bar{x}+t_kw_k=P_\lambda f(y_k)$ and $e_\lambda f(y_k)=f(\bar{x}+t_kw_k)+\frac{\lambda}{2}t_k^2\|v_k\|^2$ for all $k$ sufficiently large.
Therefore,   we have
\begin{equation}\label{lichangxi}
\begin{array}{ll}
\displaystyle\lim_{k\to \infty}\frac{f(\bar{x}+t_kw_k)-f(\bar{x})}{\frac12 t_k^2}&=\displaystyle\lim_{k\to \infty}\frac{e_\lambda f(y_k)-e_\lambda f(\bar{x})-
\frac{\lambda}{2}t_k^2\|v_k\|^2}{\frac12 t_k^2}\\
&=d^2e_\lambda f(\bar{x})(w+\lambda v)-\lambda\|v\|^2\\
&=e_\lambda(d^2f(\bar{x}\mid 0))(w+\lambda v)-\lambda\|v\|^2,
\end{array}
\end{equation}
where the second equality follows from the twice semidifferentiability of $e_\lambda f$ (cf. Lemma \ref{lem-prox-twice-epidiff} {\bf (b)}), and the third one follows from (\ref{fjjy000}).
As (\ref{lichangxi}) holds for all  $\lambda>0$   sufficiently small, we further  have
\[
\displaystyle\lim_{k\to \infty}\frac{f(\bar{x}+t_kw_k)-f(\bar{x})}{\frac12 t_k^2}=\displaystyle\lim_{\lambda\to 0+} \left(
e_\lambda(d^2f(\bar{x}\mid 0))(w+\lambda v)-\lambda\|v\|^2\right)=d^2f(\bar{x}\mid 0)(w),
\]
where the second equality follows from \cite[Theorem 1.25]{roc}. This completes the proof. \end{proof}

With Lemmas  \ref{lem-prox-regular}-\ref{lem-rest-semi} at hand, we now present our main results.
The idea in our argument is as follows: by assuming the prox-regularity of $f$ at $\bar{x}$ for 0 and the nonsingularity of $D(\partial f)(\bar{x}\mid 0)$, we observe that $u$ in (\ref{tfdf}) cannot be 0, implying by Theorem \ref{theo-relation} that $f$ satisfies the K{\L} property  at $\bar{x}$ with exponent $\frac12$; and  by further  assuming that $f$ is in addition twice epi-differentiable and  subdifferentially continuous  at $\bar{x}$ for $0$, we get from Lemma \ref{lem-rest-semi} that  $\tau$ in (\ref{tfdf}) is positive and equals to $d^2f(\bar{x}\mid 0)$, which, together with Theorem \ref{theo-good-lower-bound} (with $\theta=\frac12$) and (\ref{jbcz}), gives exact formulas for  $\partial^>g(\bar{x})$ with two alternative expressions.
By noting that, for all $\lambda>0$ sufficiently small,   $e_\lambda f$ is of class
 $\mathcal{C}^{1+}$ on some neighborhood of $\bar{x}$   and  prox-regular,  subdifferentially continuous  and twice epi-differentiable at $\bar{x}$ with the nonsingularity of $D(\nabla e_\lambda f)(\bar{x})$ being guaranteed by the nonsingularity of  $D(\partial f)(\bar{x}\mid 0)$, we are  able to    obtain the similar results for the envelopes functions  $e_\lambda f$, and to discuss some connections between the K{\L} properties of $f$ and $e_\lambda$.

\begin{theorem}\label{theo-prox-regular-twice-epidiff}
Assume that  $f:\R^n\to \overline{\R}$ is prox-regular at $\bar{x}$ for $0$. Let $g$ be given by (\ref{def-gg}). Assume  that $D(\partial f)(\bar{x}\mid 0)$ is nonsingular, i.e., 
\begin{equation}\label{fqy}
[D(\partial f)(\bar{x}\mid 0)]^{-1}(0)=\{0\}.
\end{equation}
  Then  $f$ satisfies the K{\L} property  at $\bar{x}$ with exponent $\frac12$,   which is {\bf sharp} if
\begin{equation}\label{bymm}
W:=\{w\in \R^n\mid 0<d^2f(\bar{x}\mid 0)(w)<+\infty\}\not=\emptyset.
\end{equation}
If in addition   $f$ is twice epi-differentiable and  subdifferentially continuous  at $\bar{x}$ for $0$,    we have
\begin{description}
   \item[(a)] $W\not=\emptyset$ if and only if there is some $v\in D(\partial f)(\bar{x}\mid 0)(w)$ such that $\langle v, w\rangle>0$. Moreover, we have
   \begin{equation}\label{m-deng-suofang}
       K{\L}(f,\bar{x}, \frac12)= K{\L}(\frac12 d^2f(\bar{x}\mid 0) , 0, \frac12)\geq  \sqrt{\frac{ 1}{2\,{\rm subreg}\,\partial f (\bar{x}\mid 0)}}>0,
       \end{equation}
         and
\begin{equation}\label{kuku}
\begin{array}{rl}
  \partial^>g(\bar{x})=&\displaystyle\bigcup_{w\in W} \frac{\frac12\partial [\frac12 d^2f(\bar{x}\mid 0)](w)}{\sqrt{ \frac12 d^2f(\bar{x}\mid 0)(w) }}\\
=&\left\{\displaystyle  \frac{v}{\sqrt{2\langle v, w\rangle}}\;\mid \;
v\in D(\partial f)(\bar{x}\mid 0)(w),\;
\langle v, w\rangle>0
\right\}.
\end{array}
\end{equation}
      \item[(b)] Define $\tilde{g}(x):=\sqrt{\max\{e_\lambda f(x)-e_\lambda f(\bar{x}), 0\}}$ and
      \[\widetilde{W}:= \{w\in \R^n\mid 0<e_\lambda (d^2f(\bar{x}\mid 0))(w)<+\infty\}.
      \]
      Then for all $\lambda>0$ sufficiently small, $e_\lambda f$ satisfies the K{\L} property  at $\bar{x}$ with exponent $\frac12$,   which is {\bf sharp}    if
$\widetilde{W}\not=\emptyset$ or equivalently there is  some $v\in D(\partial f)(\bar{x}\mid 0)(w)$ such that $\langle v, w\rangle+\lambda\|v\|^2>0$.
Moreover,  we have
\[
K{\L}(e_\lambda f,\bar{x}, \frac12)= K{\L}( e_\lambda(\frac12 d^2f(\bar{x}\mid 0)) , 0, \frac12)\geq \sqrt{\frac{1}{2\lambda+2\, {\rm subreg}\,\partial f (\bar{x}\mid 0)  }}>0,
\]
 and
\begin{equation}\label{mlbgongshi}
\begin{array}{rl}
& \partial^>\tilde{g}(\bar{x}) \\
&\hskip-0.65cm=\displaystyle \bigcup_{w\in \widetilde{W}} \frac{\frac12 \nabla e_\lambda (\frac12 d^2f(\bar{x}\mid 0))(w)}{\sqrt{e_\lambda( \frac12 d^2f(\bar{x}\mid 0))(w) }}  \\
=&\left\{\displaystyle \frac{v}{\sqrt{2\langle v, w\rangle+2\lambda\|v\|^2}}  \mid
v\in D(\partial f)(\bar{x}\mid 0)(w), \langle v, w\rangle+\lambda\|v\|^2>0
\right\}.
\end{array}
\end{equation}
\item[(c)] In the sense of Painlev\'{e}-Kuratowski convergence, we have
\begin{equation}\label{mlb_jixian}
\lim_{\lambda\to 0+}  \partial^>\tilde{g}(\bar{x})=\partial^>g(\bar{x}),
\end{equation}
which implies that
\begin{equation}\label{mo-danzeng}
K{\L}(e_\lambda f, \bar{x}, \frac12)\nearrow K{\L}(f, \bar{x}, \frac12) \quad\quad \mbox{as}\quad \lambda \searrow 0.
\end{equation}
\end{description}
\end{theorem}
\begin{proof}
 Suppose by contradiction that $f$ does not satisfy  the K{\L} property  at $\bar{x}$ with exponent $\frac12$, or equivalently  by Theorem \ref{theo-relation} that  $0\in \partial^>g(\bar{x})$. Then by Lemma \ref{lem-prox-regular}, we have $0\in D(\partial f)(\bar{x}\mid 0)(w)$ for some $w\in \Sph$, contradicting to our assumption (\ref{fqy}). This  indicates that  $f$ satisfies the K{\L} property  at $\bar{x}$ with exponent $\frac12$ whenever (\ref{fqy}) is fulfilled.  Whenever (\ref{bymm}) is satisfied, we get from  Theorem  \ref{theo-good-lower-bound} {\bf (b)} (with $\theta=\frac12$)  that $\partial^>g(\bar{x})\not=\emptyset$. Therefore, whenever both (\ref{fqy}) and (\ref{bymm}) are fulfilled, we get from Theorem \ref{theo-relation} that  $f$ satisfies the K{\L} property  at $\bar{x}$ with a {\bf sharp} K{\L} exponent  $\frac12$.

To show {\bf (a)}, we first get from  Theorem  \ref{theo-good-lower-bound} (with $\theta=\frac12$) that
\begin{equation}\label{paozhuanyy}
\bigcup_{w\in W } \frac{\frac12 \partial [\frac12 d^2f(\bar{x}\mid 0)](w)}{ \sqrt{ \frac12 d^2f(\bar{x}\mid 0)(w) }}\subset \partial^>g(\bar{x}).
\end{equation}
To show the reverse inclusion of (\ref{paozhuanyy}), let $u\in \partial^>g(\bar{x})$ be arbitrarily given. By Lemma \ref{lem-prox-regular},  there exist some $\tau\geq 0$,  $w\in \Sph$,   $x^k\to_f \bar{x}$ with $f(x^k)>f(\bar{x})$ and $v^k\in \partial f(x^k)$  such that
$\frac{f(x^k)-f(\bar{x})}{\frac{1}{2}\|x^k-\bar{x}\|^2}\to \tau$,   $\frac{x^k-\bar{x}}{\|x^k-\bar{x}\|}\to w$     and
\begin{equation}\label{tingyouqu}
v^k/\|x^k-\bar{x}\| \to  \sqrt{2\tau} u\in D(\partial f)(\bar{x}\mid 0)(w).
\end{equation}
Clearly,  we have $\tau\not=0$, for otherwise we  get from (\ref{tingyouqu}) that $0\in D(\partial f)(\bar{x}\mid 0)(w)$,
contradicting to  (\ref{fqy}). Therefore, we have $0<\tau<+\infty$.
Let $t^k:=\|x^k-\bar{x}\|$ and $w^k:=\frac{x^k-\bar{x}}{\|x^k-\bar{x}\|}$ for all $k$. Clearly, we have  $x^k=\bar{x}+t^kw^k$ for all $k$ and $t^k\to 0+$.  It follows from (\ref{tingyouqu}) and  Lemma \ref{lem-rest-semi} that
\begin{equation}\label{tingyouqu000}
\tau=\lim_{k\to \infty}\frac{f(x^k)-f(\bar{x})}{\frac{1}{2}\|x^k-\bar{x}\|^2}=d^2f(\bar{x}\mid 0)(w).
\end{equation}
Thus, we have  $0<d^2f(\bar{x}\mid 0)(w)<+\infty$ or in other words  $w\in W$. Moreover, it follows from  (\ref{jbcz}),  (\ref{tingyouqu}) and (\ref{tingyouqu000}) that
$
u\in  \frac{\frac12 \partial [\frac12 d^2f(\bar{x}\mid 0)](w)}{\sqrt{ \frac12 d^2f(\bar{x}\mid 0)(w) }}
$,
i.e., the reverse inclusion of (\ref{paozhuanyy}) is fulfilled. Therefore, the first equality in (\ref{kuku}) is verified. The second equality in (\ref{kuku})
follows readily from   (\ref{jbcz}) and (\ref{wjfdyjbdl}). According to  Theorem \ref{theo-good-lower-bound} (with $\theta=\frac12$) and (\ref{kuku}),  $W \not=\emptyset$  if and only if  there is some  $v\in D(\partial f)(\bar{x}\mid 0)(w)$ such that $\langle v, w\rangle>0$. In view of the definition of ${\rm subreg}\,\partial f (\bar{x}\mid 0)$ (cf. (\ref{def-smallest-norm})),  (\ref{kuku}) and the fact that $d^2f(\bar{x}\mid 0)$ is a proper function  with $d^2f(\bar{x}\mid 0)(0)=0$, we get (\ref{m-deng-suofang})  from Theorem \ref{theo-relation} {\bf (d)}  in a straightforward way. By now, all the results in {\bf (a)} are verified.

To show {\bf (b)}, we  note from Lemma \ref{lem-prox-twice-epidiff} {\bf (b)}  that for all $\lambda>0$ sufficiently small,  $e_\lambda f$ is of class
 $\mathcal{C}^{1+}$ on some neighborhood of $\bar{x}$ with $\nabla e_\lambda f(\bar{x})=0$, and is  prox-regular,  subdifferentially continuous  and twice epi-differentiable at $\bar{x}$. Moreover, we get from
Lemma \ref{lem-prox-twice-epidiff} {\bf (b)}  and (\ref{fqy})
that
$
[D(\nabla e_\lambda f)(\bar{x})]^{-1}(0)=\{0\}
$.
 Thus, for all $\lambda>0$ sufficiently small, by applying the results in {\bf (a)} to $e_\lambda f$ and by taking Lemma \ref{lem-prox-twice-epidiff}, in particular  (\ref{fjjy000})-(\ref{jbcz111}) into account, we get all the results in {\bf (b)}.

To show (\ref{mlb_jixian}), we first show
\begin{equation}\label{xblx}
\partial^>g(\bar{x})\subset \liminf_{\lambda\to 0+}  \partial^>\tilde{g}(\bar{x}).
\end{equation}
Let $u\in \partial^>g(\bar{x})$ be given arbitrarily. By (\ref{kuku}), there is some $v\in D(\partial f)(\bar{x}\mid 0)(w)$ with $\langle v, w\rangle >0$ such that
$u=\frac{1}{\sqrt{2}}\frac{v}{\sqrt{\langle v, w\rangle}}$. Then  for all $\lambda>0$ sufficiently small, we get from (\ref{mlbgongshi}) that
$\tilde{u}(\lambda):=\frac{1}{\sqrt{2}}\frac{v}{\sqrt{\langle v, w\rangle+\lambda\|v\|^2}}=\frac{u}{\sqrt{1+2\lambda\|u\|^2}}\in \partial^>\tilde{g}(\bar{x})$. 
Since $\tilde{u}(\lambda)\to u$ as $\lambda\to 0+$ and  $u\in \partial^>g(\bar{x})$ is given arbitrarily, we get (\ref{xblx}).
To show (\ref{mlb_jixian}), it suffices to show
\begin{equation}\label{sblx}
\limsup_{\lambda\to 0+}  \partial^>\tilde{g}(\bar{x})\subset\partial^>g(\bar{x}).
\end{equation}
Let $u^*\in\limsup_{\lambda\to 0+}  \partial^>\tilde{g}(\bar{x})$. By (\ref{mlbgongshi}), there exist
$\lambda_k\to 0+$ and  $v_k\in D(\partial f)(\bar{x}\mid 0)(w_k)$
with $\langle v_k, w_k\rangle+\lambda_k\|v_k\|^2>0$ for all $k$ such that
\begin{equation}\label{ongongong}
u_k:=\frac{1}{\sqrt{2}}\frac{v_k}{\sqrt{\langle v_k, w_k\rangle+\lambda_k\|v_k\|^2}}\to u^*.
\end{equation}
Clearly, we have $w_k\not=0$ for all $k$ sufficiently large, for otherwise  $u_k=\frac{1}{\sqrt{2\lambda_k}}\frac{v_k}{\|v_k\|}$ cannot converge to $u^*$.
By the positive homogeneity of $D(\partial f)(\bar{x}\mid 0)$ and by taking a subsequence if necessary, we can assume that $w_k\in \Sph$ for all $k$,  and   that
$w_k\to w^*\in \Sph$.  As $\lambda_k\to 0+$ and
$  \|u_k\|\geq \frac{1}{\sqrt{2}} \frac{\|v_k\|}{\sqrt{\|v_k\|+\lambda_k\|v_k\|^2}}\geq
\frac{1}{\sqrt{2}} \frac{1}{\sqrt{\frac{1}{\|v_k\|}+\lambda_k}}$ for all $k$,  
we claim that $\{v_k\}$ is bounded, for otherwise $\{u_k\}$ must be unbounded, contradicting to (\ref{ongongong}). So by taking a subsequence again  if necessary,
we may assume that $v_k\to v^*$. Then we have $ \langle v_k, w_k\rangle+\lambda_k\|v_k\|^2 \to  \langle v^*, w^*\rangle \geq 0$.
In view of (\ref{ongongong}), we
get $v_k=\sqrt{2}\sqrt{\langle v_k, w_k\rangle+\lambda_k\|v_k\|^2} u_k$ for all $k$. This gives  $v^*=\sqrt{2}\sqrt{\langle v^*, w^*\rangle}u^*$.
By the closedness of $D(\partial f)(\bar{x}\mid 0)$, we have $v^*\in D(\partial f)(\bar{x}\mid 0)(w^*)$. By  (\ref{fqy}), we have $v^*\not=0$ and
hence $u^*=\frac{1}{\sqrt{2}}\frac{v^*}{\sqrt{\langle v^*, w^*\rangle}}$ with
$\langle v^*, w^*\rangle>0$. According to  (\ref{kuku}), we have  $u^*\in \partial^>g(\bar{x})$.
This verifies (\ref{sblx}) and hence (\ref{mlb_jixian}).

To show (\ref{mo-danzeng}), by Theorem \ref{theo-relation} we have
 $\mbox{K}{\L}(f, \bar{x}, \frac12)= d\left(0, \partial^>g(\bar{x})\right)$  and that $\mbox{K}{\L}(e_\lambda f, \bar{x}, \frac12)=  d\left(0, \partial^>\tilde{g}(\bar{x})\right)$. In view of the set convergence in (\ref{mlb_jixian}), we get from
the pointwise convergence of distance function  \cite[Corollary 4.7]{roc} that $\mbox{K} {\L}(e_{\lambda} f, \bar{x}, \frac12)\to \mbox{K} {\L}(f, \bar{x}, \frac12) $ as
$\lambda\searrow 0$.
It remains to show the monotonicity of $\mbox{K}{\L}(e_{\lambda} f, \bar{x}, \frac12)$ as  $\lambda\searrow 0$. Let $\lambda_2>0$ be sufficiently small such that (\ref{mlbgongshi}) is fulfilled with $\lambda=\lambda_2$. Let $0<\lambda_1<\lambda_2$. Clearly,
(\ref{mlbgongshi}) is also fulfilled with $\lambda=\lambda_1$.  Let $\tilde{g}_i:=\tilde{g}$ with $\lambda=\lambda_i$ for $i=1,2$. Let $u_1\in \partial^>\tilde{g}_1(\bar{x})$ be given arbitrarily.  By (\ref{mlbgongshi}), there is some $v\in D(\partial f)(\bar{x}\mid 0)(w)$ with
$\langle v, w\rangle +\lambda_1\|v\|^2>0$ such that
$
u_1=\frac{1}{\sqrt{2}}\frac{v}{\sqrt{\langle v, w\rangle +\lambda_1\|v\|^2}}
$.
Then we have $\langle v, w\rangle +\lambda_2\|v\|^2>0$ and
$
u_2:=\frac{1}{\sqrt{2}}\frac{v}{\sqrt{\langle v, w\rangle +\lambda_2\|v\|^2}}\in \partial^>\tilde{g}_2(\bar{x})
$.
As $\|u_2\|\leq \|u_1\|$ and  $u_1\in \partial^>\tilde{g}_1(\bar{x})$ is given arbitrarily, we have
$d\left(0, \partial^>\tilde{g}_2(\bar{x})\right)\leq d\left(0, \partial^>\tilde{g}_1(\bar{x})\right)$ 
and
$
\mbox{K}{\L}(e_{\lambda_2} f, \bar{x}, \frac12)\leq  \mbox{K}{\L}(e_{\lambda_1} f, \bar{x}, \frac12).
$
That is,   $\mbox{K} {\L}(e_{\lambda} f, \bar{x}, \frac12)$ is nondecreasing as $\lambda\searrow 0$.  This completes the proof. \end{proof}

 The following corollary relates the nonsigularity of $D(\partial f)(\bar{x}\mid 0)$ with the positive definiteness of $D(\partial f)(\bar{x}\mid 0)$ or in other words the quadratic growth condition at $\bar{x}$.
\begin{corollary}\label{cor-second-order-growth}
Assume that $f:\R^n\to \overline{\R}$ is prox-regular, subdifferentially continuous  and  twice epi-differentiable at $\bar{x}$ for $0$. Consider the following conditions:
\begin{description}
\item[(i)]   $f$ satisfies the quadratic growth condition at $\bar{x}$.
\item[(ii)] The inequality  $d^2f(\bar{x}\mid 0)(w)>0$ holds for all $w\not=0$.
\item[(iii)] $C:=\{w\in \R^n\mid d^2f(\bar{x}\mid 0)(w)\leq 1\}$ is bounded.
\item[(iv)] $D(\partial f)(\bar{x}\mid 0)$ is positive definite.
\item[(v)] $[D(\partial f)(\bar{x}\mid 0)]^{-1}(0)=\{0\}$, i.e., $D(\partial f)(\bar{x}\mid 0)$ is nonsingular.
\end{description}
Then ${\bf (i)}\Longleftrightarrow {\bf (ii)}\Longleftrightarrow {\bf (iii)}\Longleftrightarrow {\bf (iv)}\Longrightarrow {\bf (v)}$  and all the conditions are equivalent in the case of  $f$ being convex.
If one of conditions {\bf (i)-(iv)} is satisfied,   then $f$ satisfies the K{\L} property  at $\bar{x}$ with exponent  $\frac12$, which is {\bf sharp}  if   $\dom [D(\partial f)(\bar{x}\mid 0)]\not=\{0\}$  or in other words  $C\not=\{0\}$.
\end{corollary}
\begin{proof}  The equivalence of {\bf (i)} and {\bf (ii)} can be found in \cite[Theorem 13.24]{roc}. The equivalence of  {\bf (ii)} and {\bf (iii)} follows readily from the lsc and the positive homogeneity  of $d^2f(\bar{x}\mid 0)$.
  The equivalence of {\bf (i)} and {\bf (iv)} can be found in  \cite[Theorem 3.7]{chnt21}. The implication from {\bf (iv)} to  {\bf (v)} is straightforward.
To show that all the conditions are equivalent by assuming that $f$ is convex,  it suffices to show ${\bf (v)}\Longrightarrow {\bf (ii)}$. Suppose by contradiction that {\bf (ii)} is not satisfied, i.e., there is some $w\not=0$ such that $d^2f(\bar{x}\mid 0)(w)\leq 0$. As $f$ is convex, we get from \cite[Proposition 13.20]{roc} that $d^2f(\bar{x}\mid 0)\geq 0$. So we actually have $d^2f(\bar{x}\mid 0)(w)=0$. This indicates  that   $w$ is a minimum of  $\frac12 d^2f(\bar{x}\mid 0)$. Then  by optimality we have  $0\in \partial [\frac12 d^2f(\bar{x}\mid 0)]=D(\partial f)(\bar{x}\mid 0)(w)$, contradicting to {\bf (v)}. Applying Theorem \ref{theo-prox-regular-twice-epidiff}, we get the remaining results immediately. \end{proof}

The following example demonstrates that, without the prox-regularity of $f$ at $\bar{x}$ for 0,     $f$ may not satisfy the K{\L} property  at $\bar{x}$ with any exponent $\theta\in [0, 1)$, even if the nonsingularity condition (\ref{fqy}), the quadratic growth condition and twice epi-differentiability  are all fulfilled.
\begin{example} Consider a function  $f:\R\to \R$  such that $f(0)=0$, $f(x)=x^2+\frac14$ if $|x|>\frac{1}{2}$, and $f(x)=x^2+\frac{1}{n}-\frac{1}{n^2}$ if $\frac{1}{n}<|x|\leq \frac{1}{n-1}$ for $n=3,4,\cdots$. It is easy to verify  that   $f$ is  continuous and twice epi-differentiable at 0 for 0 with $d^2f(\bar{x}\mid 0)=\delta_{\{0\}}$  and
\[
\partial f(x)=\left\{
\begin{array}{ll}
\R & \mbox{if}\;x= 0,\\[0.01cm] 
[2x, \infty) & \mbox{if}\;x= \frac{1}{n-1},\;n=3,4,\cdots,\\ 
(-\infty, 2x] & \mbox{if}\;x=-\frac{1}{n-1},\; n=3,4,\cdots,\\ 
\{2x\} & \mbox{otherwise}.
\end{array}
\right.
\]
Let $x^k:=\frac{1}{k-\frac12}$, $y^k:=0$, $u^k:=2x^k$ and $v^k:=\frac{1}{\sqrt{k}}$. Then we have $(x^k, u^k)\to (0, 0)$ and $(y^k, v^k)\to (0, 0)$ with
$u^k\in \partial f(x^k)$ and $v^k\in \partial f(y^k)$ for all $k$. Then for any $\rho\in \R_+$,  $\partial f+\rho I$ cannot be monotone  as
$
\lim_{k\to +\infty}\frac{(u^k-v^k)(x^k-y^k)}{(x^k-y^k)^2}=\lim_{k\to +\infty}(2+\frac{1}{2\sqrt{k}}-\sqrt{k})= -\infty$. 
By \cite[Theorem 13.36]{roc}, $f$ is not prox-regular at 0 for 0.
By definition we have
\[
D(\partial f)(0\mid 0)(w)=\left\{
\begin{array}{ll}
[2w,+\infty) &\mbox{if}\;w>0,\\
\R&\mbox{if}\;w=0,\\
(-\infty, 2w]&\mbox{if}\;w<0,
\end{array}
\right.
\]
and thus $[D(\partial f)(0\mid 0)]^{-1}(0)=\{0\}$.  Moreover,  the quadratic growth condition holds at 0 by definition. Observing that $D(\partial f)(0\mid 0)$ is positively definite, the quadratic growth condition at 0 can also obtained from \cite[Theorem 3.2]{chnt21}. But for $z^k:=\frac{1}{k-1}$ and $w^k:=2z^k\in \partial f(z^k)$, we have
$\lim_{k\to \infty}\frac{(1-\theta)w^k}{\left(f(z^k)-f(0)\right)^{\theta}}=\lim_{k\to \infty}\frac{2(1-\theta)k^\theta}{k-1}=0$ for all $\theta\in [0, 1)$,
which implies by (\ref{wcwf}) that $0\in \partial^>f^{1-\theta}(0)$ for all $\theta\in [0, 1)$.  Therefore  by Theorem \ref{theo-relation}, $f$ cannot  satisfy the K{\L} property  at $0$ with any exponent $\theta\in [0,1)$.
\end{example}

\section{Basic examples in optimization}\label{sec-basic-example}
\subsection{Smooth  functions}
\hskip-0.1cm In this subsection, we study the K{\L} property of  smooth functions of class $\mathcal{C}^2$ by applying Theorem \ref{theo-prox-regular-twice-epidiff} in a direct way.
\begin{proposition} \label{prop-nonsingular}
Assume that $f:\R^n\to \overline{\R}$ is twice continuously differentiable around  some  $\bar{x}\in \R^n$ with $\nabla f(\bar{x})=0$ and $\nabla^2 f(\bar{x})$ being nonsingular. The following hold:
\begin{description}
  \item[(a)] In the case of $\nabla^2 f(\bar{x})$ being negative definite,   $f$ satisfies the K{\L} property at $\bar{x}$ with exponent $0$, which is not {\bf sharp} as $\partial^>f(\bar{x})=\emptyset$.
    \item[(b)] In the case of $\nabla^2 f(\bar{x})$ being positive definite or indefinite, $f$ satisfies the K{\L} property at $\bar{x}$ with a  {\bf sharp} exponent of $\frac12$. Moreover, we have $K{\L} \left(f,\bar{x}, \frac12\right)=  \sqrt{\frac{\lambda_{\min}}{2}}$ with $\lambda_{\min}$ being  the smallest positive eigenvalue of $\nabla^2 f(\bar{x})$. In this case, we have 
\begin{equation}\label{vip-formular}
\partial^{>}g(\bar{x})=\left\{\frac{\sqrt{2}}{2} \frac{\nabla^2 f(\bar{x})w}{\sqrt{w^T\nabla^2 f(\bar{x})w}}\; \Big | \;w^T\nabla^2 f(\bar{x})w>0 \right\},
\end{equation}
where $g(x):=\sqrt{\max\{f(x)-f(\bar{x}), 0\}}$.
\end{description}
\end{proposition}
\begin{proof} In the case of $\nabla^2 f(\bar{x})$ being negative definite, $f$ admits a (strict) local maximum at $\bar{x}$, implying by definition that $\partial^>f(\bar{x})=\emptyset$ and hence that  $f$ satisfies the K{\L} property at $\bar{x}$ with exponent $0$, which is not {\bf sharp} by definition. In what follows we assume that $\nabla^2 f(\bar{x})$ is positive definite or indefinite, in which case $W:=\{w\mid 0<\langle w, \nabla^2f(\bar{x})w\rangle<+\infty\}\not=\emptyset$.   As $f$ is assumed to be twice continuously differentiable around      $\bar{x}$ with $\nabla f(\bar{x})=0$, $f$ is   fully amenable (hence prox-regular and  twice epi-differentible) and subdifferentially continuous at $\bar{x}$ for 0 with
$$
d^2f(\bar{x}\mid 0)(w)=\langle w, \nabla^2f(\bar{x})w\rangle \quad\mbox{and}\quad D(\partial f)(\bar{x}\mid 0)(w)=\nabla^2f(\bar{x})w \quad\forall w\in \R^n.
$$
By virtue of Theorem \ref{theo-prox-regular-twice-epidiff}, we get all the results except for the equality $\mbox{K}{\L} \left(f,\bar{x}, \frac12\right)=  \sqrt{\frac{\lambda_{\min}}{2}}$. To this end,  we apply the spectral decomposition theorem for a
real symmetric matrix to  find an orthogonal matrix $U$ such that
$
\nabla^2 f(\bar{x})=U \Lambda U^T$,
where  $\Lambda$ is a diagonal matrix with its diagonal entries  $\lambda_i$'s  being the eigenvalues of  $\nabla^2 f(\bar{x})$.
Let $w\in \R^n$  be any vector such that
$w^T\nabla^2 f(\bar{x})w>0$ and let
$ v:=\frac{\sqrt{2}}{2} \frac{\nabla^2 f(\bar{x})w}{\sqrt{w^T\nabla^2 f(\bar{x})w}} $.
In terms of $y:=U^Tw$, we have $w^T\nabla^2 f(\bar{x})w>0$ if and only if $\sum_{j=1}^n\lambda_jy_j^2>0$,  and moreover
$\|v\|^2=\displaystyle\frac{1}{2}\frac{\sum_{i=1}^n\lambda_i^2y_i^2}{\sum_{j=1}^n\lambda_jy_j^2}=\displaystyle\frac{1}{2}\sum_{i=1}^n \frac{\lambda_iy_i^2}{\sum_{j=1}^n\lambda_jy_j^2}\lambda_i$.
As $\sum_{i=1}^n \frac{\lambda_iy_i^2}{\sum_{j=1}^n\lambda_jy_j^2}=1$, we have $\|v\|^2\geq \frac{1}{2}\lambda_{\min}$ and the equality holds  whenever  $y_i=0$ for all $i$ with $\lambda_i\not=\lambda_{\min}$. This verifies the equality  $\mbox{K}{\L} \left(f,\bar{x}, \frac12\right)=  \sqrt{\frac{\lambda_{\min}}{2}}$  when (\ref{vip-formular}) and Theorem \ref{theo-relation} are taken into account.  The proof is   completed.  \end{proof}


\subsection{The  pointwise max  of finitely many smooth functions}
In this subsection, we consider a function  $f:=\max\{f_1, \cdots, f_m\}$ for functions $f_i$ of class $\mathcal{C}^2$ on $\R^n$ and a point $\bar{x}\in \R^n$ with $0\in \partial f(\bar{x})$.  According to \cite[Example 10.24 and Theorems 13.16]{roc},  $f$ is  fully amenable everywhere  and thus  prox-regular, twice epi-differentiable, and subdifferentially continuous at $\bar{x}$ for  0 (hence  Theorem \ref{theo-prox-regular-twice-epidiff} is applicable) with
\begin{equation}\label{second-order-sub}
d^2f(\bar{x}\mid 0)(w)= \max_{\lambda \in\Lambda(\bar{x})}  \{\sum_{i\in I(\bar{x})} \lambda _i\langle\nabla^2f_i(\bar{x})w,\;w\rangle \}+ \delta_{\mathcal{V}(\bar{x})}(w),
\end{equation}
where
$
I(\bar{x}):=\{i\mid f_i(\bar{x})=f(\bar{x})\}
$,  $\mathcal{V}(\bar{x}):=\{w\in \R^n\mid \langle \nabla f_i(\bar{x}), \;w\rangle\leq 0\;\forall i\in I(\bar{x})\}$  and
$$
\Lambda(\bar{x}):= \{\lambda\in \R^m\mid \lambda_i\geq 0\;\forall i\in I(\bar{x}),\; \sum_{i\in I(\bar{x}) }\lambda_i=1, \sum_{i\in I(\bar{x}) }\lambda_i\nabla f_i(\bar{x})=0 \}.
$$
Moreover, it has been shown in  \cite[Theorem 2]{pr94} that 
$$
D(\partial f)(\bar{x}\mid 0)(w)= \{\sum_{i\in I(\bar{x})} \lambda _i \nabla^2f_i(\bar{x})w\mid \lambda\in \Lambda(\bar{x}, w) \}+N_{\mathcal{V}(\bar{x})}(w),$$
where $\Lambda(\bar{x}, w)$ is the optimal solution set of the linear programming problem
\begin{equation}\label{primal}
 \max_{\lambda} \quad\displaystyle\sum_{i\in I(\bar{x})} \lambda _i\langle\nabla^2f_i(\bar{x})w,\;w\rangle\quad \mbox{subject to}\quad  \lambda\in \Lambda(\bar{x}), 
\end{equation}
whose dual can be formulated as
\begin{equation}\label{dual}
 \min_{(\kappa, z)} \quad \kappa \quad \mbox{subject to}\quad \kappa-\langle \nabla f_i(\bar{x}), z \rangle-\langle \nabla^2 f_i(\bar{x})w, w \rangle\geq 0\quad \forall i \in I(\bar{x}).
\end{equation}
Then it follows from  the complementary slackness conditions and characterization of optimality for linear programming problems \cite[Theorem 2.7.3 and its  Corollary 3]{bss06} that, for any feasible  $\lambda$  to (\ref{primal})  and any feasible $(\kappa, z)$ to (\ref{dual}), $\lambda$ and $(\kappa, z)$ are, respectively, optimal  to (\ref{primal}) and  (\ref{dual}) if and only if $\lambda_i(\kappa-\langle \nabla f_i(\bar{x}), z \rangle-\langle \nabla^2 f_i(\bar{x})w, w \rangle)=0$ for all $i\in I(\bar{x})$.
Therefore, for any $(w, v)\in \R^n \times \R^n$,  we have $v\in D(\partial f)(\bar{x}\mid 0)(w)$ if and only if  there is some $(\kappa, \lambda, \beta,   z)\in \R\times \R^m\times \R^m\times \R^n$ such that
$v=\sum_{i\in I(\bar{x})}\left[\lambda_i\nabla^2f_i(\bar{x})w+\beta_i\nabla f_i(\bar{x})\right]$ with 
\begin{equation}\label{ys1}
\sum_{i\in I(\bar{x})}\lambda_i=1, \quad \sum_{i\in I(\bar{x})}\lambda_i\nabla f_i(\bar{x})=0,
\end{equation}
\begin{equation}\label{ys3}
\left.
\begin{array}{r}
\beta_i\geq 0,\;\langle \nabla f_i(\bar{x}), w\rangle\leq 0,\; \beta_i\,\langle \nabla f_i(\bar{x}), w\rangle=0\\
\lambda_i\geq 0,\;\kappa-\langle \nabla f_i(\bar{x}), z \rangle-\langle \nabla^2 f_i(\bar{x})w, w \rangle\geq 0\\
\lambda_i\,[\kappa-\langle \nabla f_i(\bar{x}), z \rangle-\langle \nabla^2 f_i(\bar{x})w, w \rangle\,]=0
\end{array}
\right\}\quad \forall i\in I(\bar{x}).
\end{equation}
Moreover, for any $(\kappa, \lambda, \beta, w, z)\in \R\times \R^m\times \R^m\times \R^n\times \R^n$ satisfying (\ref{ys1}) and (\ref{ys3}), we get  $d^2f(\bar{x}\mid 0)(w)=\kappa$. In view of these formulas and results just recalled, we  present the following proposition by applying Theorem \ref{theo-prox-regular-twice-epidiff} to $f$ in a straightforward way.

\begin{proposition}
Consider a function  $f:=\max\{f_1, \cdots, f_m\}$ for functions $f_i$ of class $\mathcal{C}^2$ on $\R^n$ and a point $\bar{x}\in \R^n$ with $0\in \partial f(\bar{x})$.
Assume that one of the following equivalent conditions holds: 
\begin{description}
\item[(a)] $D(\partial f)(\bar{x}\mid 0)^{-1}(0)=\{0\}$.
\item[(b)] For each $\bar{w}\not=0$ with  $d^2f(\bar{x}\mid 0)(\bar{w})=0$,  the following implication holds:  
\begin{equation}\label{max-kl-half}
\left.
\begin{array}{c}
\displaystyle\sum_{i\in I(\bar{x})}\lambda_i\left( \begin{array}{c}
\nabla^2f_i(\bar{x})\bar{w}\\
\nabla f_i(\bar{x})
\end{array}\right)+\displaystyle\sum_{i\in I(\bar{x})}\beta_i\left( \begin{array}{c}
\nabla f_i(\bar{x})\\
0
\end{array}\right) =\left( \begin{array}{c}
0\\
0
\end{array}\right) \\[0.5cm]
\lambda_{I(\bar{x})}\geq 0,\quad \beta_{I(\bar{x})}\geq 0\quad 
\end{array}
\right\}
\Longrightarrow \lambda_{I(\bar{x})}=0.
\end{equation}
\item[(c)]  The optimal value of the following smooth optimization  problem is positive:
\begin{eqnarray*}
\min_{(\kappa, \lambda, \beta, w, z)}&& \frac{1}{2}\| \sum_{i\in I(\bar{x})}\left[\lambda_i\nabla^2f_i(\bar{x})w+\beta_i\nabla f_i(\bar{x})\right]\|^2\\
\mbox{subject to} && (\ref{ys1})-(\ref{ys3}),\; \|w\|^2=1.
\end{eqnarray*}
\end{description}
Then  $f$ satisfies the K{\L} property at $\bar{x}$ with exponent $\frac12$, which is {\bf sharp} if   there is some $w\in \mathcal{V}(\bar{x})$ with $d^2f(\bar{x}\mid 0)(w)>0$ or equivalently  there is some $(\kappa, \lambda, \beta, w, z)$ satisfying (\ref{ys1}-\ref{ys3}) with $\kappa>0$.
Moreover,   $K{\L}(f, \bar{x}, \frac12)$ is  the square root of the optimal value of the smooth optimization  problem
\begin{eqnarray*}
  \quad \min_{(\kappa, \lambda, \beta, w, z)} && \displaystyle \frac{1}{2}\frac{\| \sum_{i\in I(\bar{x})}\left[\lambda_i\nabla^2f_i(\bar{x})w+\beta_i\nabla f_i(\bar{x})\right]\|^2}{\kappa}\\
  \mbox{subject to} && (\ref{ys1})-(\ref{ys3}), \; \|w\|^2=1,\; \kappa>0.
\end{eqnarray*}
\end{proposition}
\subsection{The $\ell_1$ regularized functions}
In this subsection, we study the K{\L} property of  the  $\ell_1$ regularized function
\begin{equation}\label{l1-regu}
F(x):=f(x)+\mu\|x\|_1,
\end{equation}
where $f:\R^n\to \R$ is assumed to be of class $\mathcal{C}^2$ and $\mu>0$. Note that $F$ in the form of (\ref{l1-regu}) includes the $\ell_{1-2}$ minimization function studied in \cite{apx17, ylhx15} as  a special case. Now consider some $\bar{x}\in \R^n$ with  $0\in \partial F(\bar{x})$.  By \cite[Theorem 10.24 {\bf (g)}]{roc},   $F$ is  a fully amenable function, and thus  prox-regular, twice epi-differentiable, and subdifferentially continuous at $\bar{x}$ for  0  \cite[Theorem 13.14 and Proposition 13.32]{roc}.  To apply   Theorem \ref{theo-prox-regular-twice-epidiff} to $F$, we recall from the literature some formulas and results relating to $D(\partial F)(\bar{x}\mid 0)$ and $d^2F(\bar{x}\mid 0)$.
 By some calculus rules of subgradients, we deduce from  $0\in \partial F(\bar{x})$ that  each $j\in \{1,\cdots, n\}$ falls  into one of the following distinct index sets:
\begin{equation}\label{four-indexs}
\begin{array}{ll}
I:=\{i\mid \bar{x}_i=0, \;-\mu<(\nabla f(\bar{x}))_i<\mu\},& \hskip-0.3cm \;\;\;K^+:=\{i\mid \bar{x}_i=0,\;(\nabla f(\bar{x}))_i=\mu\},\\
J:=\{i\mid \bar{x}_i\not=0,\;(\nabla f(\bar{x}))_i+\mu \sign(\bar{x}_i)=0\},& K^-:=\{i\mid \bar{x}_i=0,\;(\nabla f(\bar{x}))_i=-\mu\},
\end{array}
\end{equation}
 and that $K^+\cup K^-=\emptyset$ if in particular $0\in \ri \partial F(\bar{x})$. According to \cite[Proposition 3.1 and Theorem 3.1]{cln22},  we have $v\in D(\partial F)(\bar{x}\mid 0)(w)$ if and only if
\begin{equation}\label{chigua}
\left.
\begin{array}{rcrl}
w_j=0&\forall j\in  I,&v_j-(\nabla^2 f(\bar{x})w)_j=0& \forall j\in J\\
w_j\leq 0,&v_j-(\nabla^2 f(\bar{x})w)_j\geq 0,& w_j(v_j-(\nabla^2 f(\bar{x})w)_j)=0 & \forall j\in K^+\\
w_j\geq 0,& v_j-(\nabla^2 f(\bar{x})w)_j\leq 0,& w_j(v_j-(\nabla^2 f(\bar{x})w)_j)=0 & \forall  j\in K^-
\end{array}
\right\}.
\end{equation}
By virtue of (\ref{wjfdyjbdl}), we get  $d^2F(\bar{x}\mid 0)(w)=\langle \nabla^2 f(\bar{x})w, w\rangle +\delta_{\mathcal{V}(\bar{x})}(w)$  for all $w\in \R^n$, where
\begin{equation}\label{gj-cone}
\mathcal{V}(\bar{x}):=\{w\in \R^n\mid w_j=0\;\forall j\in I,\;w_j\leq 0\;\forall j\in K^+,\;w_j\geq 0\;\forall j\in K^-\}.
\end{equation}
In view of the  results just recalled, we  present the following proposition by applying Theorem \ref{theo-prox-regular-twice-epidiff} to $F$ in a straightforward way,  of which the proof is omitted.
\begin{proposition}
\hskip-0.2cm Consider the  $\ell_1$-regularized function $F$ given by  (\ref{l1-regu})   and some $\bar{x}$ with $0\in \partial F(\bar{x})$. Let $I,J,K^+,K^-$ be given by (\ref{four-indexs}) and $\mathcal{V}(\bar{x})$ be given by (\ref{gj-cone}).
Assume that  $D(\partial F)(\bar{x}\mid 0)^{-1}(0)=\{0\}$ or equivalently that 
\begin{equation}\label{tjb}
\left.
\begin{array}{rcrl}
w_j=0&\forall j\in  I,&(\nabla^2 f(\bar{x})w)_j=0 & \forall j\in J\\
w_j\leq 0,&  (\nabla^2 f(\bar{x})w)_j\leq 0,& w_j (\nabla^2 f(\bar{x})w)_j=0  & \forall j\in K^+\\
w_j\geq 0,&  (\nabla^2 f(\bar{x})w)_j\geq 0,& w_j (\nabla^2 f(\bar{x})w)_j=0  & \forall j\in K^-\\
\end{array}
\right\}
\Longrightarrow
w=0.
\end{equation}
Then  $F$ satisfies the K{\L} property at $\bar{x}$ with exponent $\frac12$, which is {\bf sharp} if   there is some $w\in \mathcal{V}(\bar{x})$ with $\langle \nabla^2 f(\bar{x})w, w\rangle>0$.
Moreover,    $K{\L}(F, \bar{x}, \frac12)$ is equal to  the square root of the optimal value of the smooth optimization  problem
\begin{equation}\label{l1-opt}
  \min_{(w, v)}\quad \displaystyle \frac{1}{2}\frac{\|v\|^2}{\langle v, w\rangle} \quad \mbox{subject to}\quad  (\ref{chigua}),\quad \|w\|^2=1,\quad \langle v, w\rangle>0.
\end{equation}
In particular when $0\in \ri \partial F(\bar{x})$, we have $K^+\cup K^-=\emptyset$. In this case, condition (\ref{tjb}) reduces to the nonsingularity of the submatrix $[\nabla^2 f(\bar{x})_{i, j}]_{i,j\in J}$, and to solve (\ref{l1-opt}) is to find the smallest positive eigenvalue of  $[\nabla^2 f(\bar{x})_{i, j}]_{i,j\in J}$.
\end{proposition}

\subsection{The $\ell_p$  regularized least square functions}
In this subsection, we  study  the K{\L}  property of  the $\ell_p$ ($0<p<1$) regularized least square function
\begin{equation}\label{rp}
F(x):=\|Ax-b\|_2^2+\mu\|x\|_p^p\quad\quad \forall x\in \R^n,
\end{equation}
where $A\in \R^{m\times n}$, $b\in \R^m$, $\mu >0$ and $\|x\|_p^p:=\displaystyle\sum_{i=1}^n |x_i|^p$. To avoid the technical verification of the prox-regularity and the twice epi-differentiability of $F$, we will not apply Theorem \ref{theo-prox-regular-twice-epidiff} in a direct way. Instead, we will apply Theorem \ref{theo-relation} and Proposition \ref{prop-nonsingular} by converting the study of the  K{\L} property of
$F$   to that of a restricted smooth function.  To begin with, we get the formula  for $\partial F(x)$  by using the  calculus rules of subgradients \cite[Exercise 8.8 and Proposition 10.5]{roc}  as follows:
\begin{equation}\label{lp-subg}
\partial F(x)\hskip-0.05cm =\hskip-0.05cm \left\{2A^T(Ax-b)+\mu p v \hskip-0.05cm \mid \hskip-0.05cm v_i={\rm sign}(x_i)|x_i|^{p-1} \forall i\in \supp(x), v_i\in \R\;\forall i\not\in \supp(x)\right\}
\end{equation}
by virtue of which, and due to $|t|^{p-1}\to +\infty$ as $t\to 0$ with $t\not=0$,  we further  have  the equality
\begin{equation}\label{lp-subg000}
\limsup_{y\to x,\;y_{i_0}\not=0}\partial F(y)=\emptyset\quad \forall i_0\not\in \supp(x).
\end{equation}
Here $\supp(x):=\{i\mid x_i\not=0\}$ denotes the support index set of $x$.
 For each nonempty  $I\subset\{1,\cdots, n\}$, let  $A_I:=[A_i]_{i\in I}$ with $A_i$ being the $i$-th column of $A$  and define  $F_I:\R^{|I|}\to \R$ by
 \begin{equation}\label{fi}
 F_I(y):=\|A_Iy-b\|^2+\mu\|y\|_p^p.
 \end{equation}
Then for each $\bar{x}\in \R^n\backslash \{0\}$,  we have in terms of $I:=\supp(\bar{x})$, $\bar{x}_I:=(\bar{x}_i)_{i\in I}$, $v_I:=(v_i)_{i\in I}$, $G(x):=(\max\{F(x)-F(\bar{x}), 0\})^{1-\theta}$, and $G_I(y)=(\max\{F_I(y)-F_I(\bar{x}_I), 0\})^{1-\theta}$, 
\begin{equation}\label{lp-outer-000}
\partial^>G(\bar{x})=\{v\in \R^n\mid v_I\in \partial^>G_I(\bar{x}_I)\}\quad \forall \theta\in [0, 1),
\end{equation}
which follows readily from  the definition of outer limiting subdifferential sets by taking (\ref{wcwf}),  (\ref{lp-subg}) and (\ref{lp-subg000}) into account.
Moreover,   we have  $\partial^>F(0)=\emptyset$ due to (\ref{lp-subg000}). Therefore, $F$ satisfies the K{\L} property at $\bar{x}=0$ with exponent $0$, which is not {\bf sharp}.
With the help of Theorem \ref{theo-relation} and the identity (\ref{lp-outer-000}),   the study of the K{\L} property of
$F$ at some  $\bar{x}\in \R^n\backslash \{0\}$ can be converted to that of $F_I$ (with $I=\supp(\bar{x})$) at $\bar{x}_I$,
around which $F_I$ is of class $\mathcal{C}^\infty$. We demonstrate this idea in the following proposition.
\begin{proposition}
Consider the $\ell_p$ regularized least square function $F$ given by (\ref{rp}) and some  $\bar{x}\in \R^n\backslash\{0\}$. The following hold in terms of  $I:=\supp(\bar{x})$, $F_I$ defined by (\ref{fi}), and  $\bar{x}_I:=(\bar{x}_i)_{i\in I}$:
\begin{description}
\item[(a)]  $ K\L(F, \bar{x}, \theta)=K\L(F_I, \bar{x}_I, \theta)$ for all $\theta\in [0, 1)$.
\item[(b)] If $\nabla F_I(\bar{x}_I)=0$ and  $\nabla^2 F_I(\bar{x}_I)$  is negative definite, then $\partial^>F(\bar{x})=\emptyset$ and  $F$ satisfies the K{\L} property at $\bar{x}$ with exponent 0, which is not {\bf sharp}.
\item[(c)]  If $\nabla F_I(\bar{x}_I)=0$ and  $\nabla^2 F_I(\bar{x}_I)$  is  positive definite or indefinite,   then $F$ satisfies the K{\L} property at $\bar{x}$ with a {\bf sharp} exponent $\frac12$,   and in this case,   $K\L(F, \bar{x}, \frac12)=\sqrt{\frac{\lambda_{\min}}{2}}$ with $\lambda_{\min}$ being  the smallest positive eigenvalue of $\nabla^2 F_I(\bar{x}_I)$.
\end{description}
\end{proposition}

\begin{remark}
According to \cite[Theorem 1]{hlmy2021}, any $\bar{x}\in \R^n\backslash\{0\}$ is a local minimum of   $F$ if and only if $\nabla F_I(\bar{x}_I)=0$ and  $\nabla^2 F_I(\bar{x}_I)$  is  positive definite. This suggests that  $F$ satisfies the K{\L} property at any  nonzero local minimum with a {\bf sharp} exponent $\frac12$. See \cite{bllzd22}  for more details on this result.  
\end{remark}

\bibliographystyle{siamplain}

\begin{thebibliography}{99}
\bibitem{apx17} M. Ahn, J.S. Pang, J. Xin, Difference-of-convex learning: directional stationarity, optimality, and sparsity, SIAM J. Optim., 27(2017), pp. 1637-1665.
\bibitem{ag08} F.J. Arag\'on Artacho and M.H. Geoffroy, Characterization of metric regularity of subdifferentials, J. Convex Anal., 15(2008), pp. 365-380.
\bibitem{ag14} F.J. Arag\'on Artacho and M.H. Geoffroy, Metric subregularity of the convex subdifferential in Banach spaces, J. Nonlinear Convex Anal., 15(2014), pp. 35-47.
\bibitem{ab} H. Attouch, J. Bolte, On the convergence of the proximal algorithm for nonsmooth functions involving analytic features, Math. Program., 116(2009), pp. 5-16.
\bibitem{abrs} H. Attouch, J. Bolte, P. Redont, A. Soubeyran, Proximal alternating minimization and projection methods for nonconvex problems: an approach based on the Kurdyka-{\L}ojasiewicz inequality, Math. Oper. Res., 35(2010), pp. 438-457.
\bibitem{abs} H. Attouch, J. Bolte, B.F. Svaiter, Convergence of descent methods for semi-algebraic and tame problems: proximal algorithms, forward-backward splitting, and regularized Gauss-Seidel methods, Math. Program., 137(2013), pp. 91-129.
\bibitem{bllzd22} S.X. Bai, M.H. Li, C.W. Lu, D.L. Zhu, S. Deng, The equivalence of three types of error bounds for weakly and approximately convex functions, J. Optim. Theory Appl., 194(2022), pp. 220-245.
\bibitem{bss06} M. S. Bazaraa, H. D. Sherali, C.M. Shetty, Nonlinear Programming: Theory and Algorithms, John Wiley \& Sons, 2006.
\bibitem{cln21}  J.Y. Bello-Cruz,, G. Li, and  T.T.A. Nghia, On the Q-linear convergence of forward-backward splitting method. Part I: Convergence analysis, J. Optim. Theory Appl., 188 (2021), pp. 378-401.
\bibitem{cln22} Y. Bello-Cruz, G.Y. Li, T.T.A. Nghia, Quadratic growth conditions and uniqueness of optimal solution to Lasso, J. Optim. Theory Appl., 194(2022), pp. 167-190.
\bibitem{bot2} J. Bolte, A. Daniilidis, A. Lewis, The {\L}ojasiewicz inequality for nonsmooth subanalytic functions with applications to subgradient dynamical systems, SIAM J. Optim., 17(2007), pp. 1205-1223.
\bibitem{bdls07} J. Bolte, A. Daniilidis, A. Lewis, M. Shiota, Clarke subgradients of stratifiable functions, SIAM J. Optim., 18(2007), pp. 556-572.
\bibitem{bdlm} J. Bolte, A. Daniilidis, O. Ley, L. Mazet, Characterizations of {\L}ojasiewicz inequalities: Subgradient flows, talweg, convexity, Trans. Amer. Math. Soc., 362(2010), pp. 3319-3363.
\bibitem{bnpb} J. Bolte, T.P. Nguyen, J. Peypouquet, W.S. Bruce, From error bounds to the complexity of first-order descent methods for convex functions, Math. Program., 165(2017), pp. 471-507.
\bibitem{bs00} J.F. Bonnans, and A. Shapiro, Perturbation Analysis of Optimization Problems. Springer, New York, 2000.
\bibitem{ca2014} M.J. C\'{a}novas, M.A. L\'{o}pez, J. Parra, F.J. Toledo, Calmness of the feasible set mapping for linear inequality systems, Set-Valued Var. Anal., 22(2014), pp. 375-389.
\bibitem{ca2016} M.J. C\'{a}novas, R. Henrion, M.A. L\'{o}pez, J. Parra, Outer limit of subdifferentials and calmness moduli in linear and nonlinear programming, J. Optim. Theory Appl., 169(2016), pp. 925-952.
\bibitem{chnt21} N.H. Chieu, L.V. Hien, T.T.A. Nghia, H.A. Tuan, Quadratic growth and strong metric subregularity of the subdifferential via subgradient graphical derivative, SIAM J. Optim, 31(2021), pp. 545-568.
\bibitem{comi91} R. Cominetti, On pseudo-differentiability, Trans. Amer. Math. Soc., 324(1991), pp. 843-847.
\bibitem{dr14} A.L. Dontchev and R.T. Rockafellar, Implicit Functions and Solution Mappings: A View from Variational Analysis, 2nd ed., Springer, New York, 2014.
\bibitem{di15} D. Drusvyatskiy and A D. Ioffe, Quadratic growth and critical point stability of semialgebraic functions, Math. Program., 153(2015), pp. 635-653.
\bibitem{dl18} D. Drusvyatskiy and A.S. Lewis, Error bounds, quadratic growth, and linear convergence of proximal methods, Math. Oper. Res., 43(2018), pp. 919-948.
\bibitem{dmn14} D. Drusvyatskiy, B.S. Mordukhovich, and T. T. A. Nghia, Second-order growth, tilt stability, and metric regularity of the subdifferential, J. Convex Anal., 21(2014), pp. 1165-1192.
\bibitem{ers} A. Eberhard, V. Roshchina, T. Sang, Outer limits of subdifferentials for min-max type functions, Optim., 68(2019), pp. 1391-1409.
\bibitem{fgp} P. Frankel, G. Garrigos, J. Peypouquet, Splitting methods with variable metric for Kurdyka-{\L}ojasiewicz functions and general convergence rates, J. Optim. Theory Appl., 165(2015), pp. 874-900.
\bibitem{knt10} A.Y. Kruger, H.V. Ngai, M. Th\'{e}ra, Stability of error bounds for convex constraint systems in Banach spaces, SIAM J. Optim., 20(2010), pp. 3280-3296.
\bibitem{kpyz19} A.Y. Kruger, M. A. L\'{o}pez, X.Q. Yang, J.X. Zhu. H\"{o}lder error bounds and H\"{o}lder calmness with applications to convex semi-infinite optimization. Set-Valued Var. Anal, 27(2019), pp. 995-1023.
\bibitem{ku98} K. Kurdyka, On gradients of functions definable in o-minimal structures, Ann. Inst. Fourier, 48(1998), pp. 769-783.
\bibitem{hlmy2021} Y.H. Hu, C. Li, K.W. Meng, X.Q. Yang, Linear convergence of inexact descent method and inexact proximal gradient algorithms for lower-order regularization problems, J. Global Optim., 79(2021), pp. 853-883.
\bibitem{io91}  A.D. Ioffe, Variational  analysis  of  a  composite  function:  A  formula  for  the  lower  second-order epi-derivative, J. Math. Anal. Appl., 160(1991), pp. 379-405.
\bibitem{io2008} A.D. Ioffe, J.V. Outrata, On metric and calmness qualification conditions in subdifferential calculus, Set-Valued Anal., 16(2008), pp. 199-227.
\bibitem{lp2016} G.Y. Li, T.K. Pong, Douglas-Rachford splitting for nonconvex optimization with application to nonconvex feasibility problems, Math. Program., 159(2016), pp. 371-401.
\bibitem{lp2018} G.Y. Li, T.K. Pong, Calculus of the exponent of Kurdyka-{\L}ojasiewicz inequality and its applications to linear convergence of first-order methods, Found. Comput. Math., 18(2018), pp. 1199-1232.
\bibitem{lmp2015} G.Y. Li, B.S. Mordukhovich, T.S. Pham, New fractional error bounds for polynomial systems with applications to H\"{o}lderian stability in optimization and spectral theory of tensors, Math. Program., 153(2015), pp. 333-362.
\bibitem{lmy} M.H. Li, K.W. Meng, X.Q. Yang, On error bound moduli for locally Lipschitz and regular functions, Math. Program., 171(2018), pp. 463-487.
\bibitem{lmy20} M.H. Li, K.W. Meng, X.Q. Yang, On far and near ends of closed and convex sets, J. Convex Anal., 27(2020), pp. 407-421.
\bibitem{lws} H. Liu, W. Wu, A.M.-C. So, Quadratic optimization with orthogonality constraints: explicit {\L}ojasiewicz exponent and linear convergence of line-search methods, ICML, 48(2016), pp. 1158-1167.
\bibitem{lo63} S. {\L}ojasiewicz, Une propri\'{e}t\'{e} topologique des sous-ensembles analytiques r\'{e}els, in Les \'{E}quations aux D\'{e}riv\'{e}es Partielles, \'{E}ditions du Centre National de la Recherche Scientifique, Paris, 117(1963), pp. 87-89.
\bibitem{mlyy21} K.W. Meng, M.H. Li, W.F. Yao, X.Q. Yang, Lipschitz-like property relative to a set and the generalized Mordukhovich criterion, Math. Program., 189(2021), pp. 455-489.
\bibitem{mms22} A. Mohammadi, B.S. Mordukhovich, M.E. Sarabi, Variational analysis of composite models with applications to continuous optimization, Math. Oper. Res., 47(2022), pp. 397-426.
\bibitem{ms20} A. Mohammadi and M.E. Sarabi, Twice epi-differentiability of extended-real-valued functions with applications in composite optimization, SIAM J. Optim., 30(2020), pp. 2379-2409
\bibitem{mor05} B.S. Mordukhovich, N.M. Nam, Subgradient of distance functions with applications to Lipschitzian stability, Math. Program., 104(2005), pp. 635-668.
\bibitem{mor} B.S. Mordukhovich, Variational Analysis and Generalized Differentiation I: Basic Theory. Springer, 2 edition, 2013.
\bibitem{pr94} R.A. Poliquin, R.T. Rockafellar, Proto-derivative formulas for basic subgradient mappings in mathematical programming, Set-Valued Anal., 2(1994), pp. 275-290.
\bibitem{pr96} R.A. Poliquin and R.T. Rockafellar, Prox-regular functions in variational analysis, Trans. Amer. Math. Soc., 348(1996), pp. 1805-1838.
\bibitem{roc85} R.T. Rockafellar, Maximal  monotone  relations  and  the  second  derivatives  of  nonsmoothfunctions, Ann. Inst. H. Poincar\'e Anal. Non Lin\'eaire, 2(1985), pp. 167-184.
\bibitem{roc89} R.T. Rockafellar, Second-order  optimality  conditions  in  nonlinear  programming  obtained by way of epi-derivatives, Math. Oper. Res., 14(1989), pp. 462-484.
\bibitem{roc} R.T. Rockafellar, R.J.-B. Wets, Variational Analysis, Springer, Berlin, 1998.
\bibitem{ww22} X. Wang, Z. Wang, The exact modulus of the generalized Kurdyka-{\L}ojasiewicz property, Math. Oper. Res., 47(2022), pp. 2765-2783.
\bibitem{wpb} Y.Q. Wu, S.H. Pan, S.J. Bi, Kurdyka-{\L}ojasiewicz property of zero-norm composite functions, J. Optim. Theory Appl., 188(2021), pp. 94-112.
\bibitem{yz} J.C. Yao, X.Y. Zheng, Error bound and well-posedness with respect to an admissible function, Appl. Anal., 95(2016), pp. 1070-1087.
\bibitem{ylhx15} P. Yin, Y. Lou, Q. He, J. Xin, Minimization of $\ell_{1-2}$ for compressed sensing, SIAM J. Sci. Comput., 37(2015), A536-A563.
\bibitem{ylp} P.R. Yu, G.Y. Li, T.K. Pong, Kurdyka-{\L}ojasiewicz exponent via inf-projection, Found. Comput. Math., 22(2022), pp. 1171-1217.
\bibitem{zlx16} J. Zeng, S. Lin, Z. Xu, Sparse regularization: convergence of iterative jumping thresholding algorithm, IEEE T. Signal Proces., 64(2016), pp.5106-5118.
\end{thebibliography}

\end{document}